\documentclass[3p,times]{elsarticle}
\makeatletter
\def\ps@pprintTitle{%
	\let\@oddhead\@empty
	\let\@evenhead\@empty
	\def\@oddfoot{}
	\let\@evenfoot\@oddfoot
}
\makeatother
\usepackage{enumerate}
\usepackage{amsmath,amsthm,verbatim,amssymb,amsfonts,amscd,graphicx}
\usepackage{hyperref}
\usepackage{mathrsfs}
\usepackage{tikz}
\usepackage{tikz-cd}
\usepackage{float}
\usepackage{hyperref}
\usepackage[all]{xy}
\usepackage{multicol}
\usepackage[mathscr]{euscript} 
\usepackage{ mathrsfs }

\newcommand{\n}{\mathfrak{n} }
\newcommand{\m}{\mathfrak{m} }

\providecommand{\e}{{\mathcal E}}
\providecommand{\D}{{\mathcal D}}
\newcommand{\Z}{\mathbb{Z} }
\newcommand{\N}{\mathbb{N} }
\newcommand{\TT}{\mathcal{T}}

\newcommand{\QQ}{\mathbb{Q} }

\newcommand{\FF}{\mathcal{F}}

\newcommand{\rt}{\rightarrow}
\newcommand{\lrt}{\longrightarrow}
\newcommand{\ov}{\overline}

\newcommand{\coker}{\operatorname{coker}}
\newcommand{\Ass}{\operatorname{Ass}}
\newcommand{\chars}{\operatorname{char}}
\newcommand{\depth}{\operatorname{depth}}

\newcommand{\gr}{\operatorname{gr}}

\newcommand{\hgt}{\operatorname{ht}}

\providecommand\Mod{\text{\rm Mod}}

\newcommand{\injdim}{\operatorname{injdim}}

\newcommand{\bideg}{\operatorname{bideg}}

\newcommand{\xdeg}{\operatorname{x-deg}}
\newcommand{\ydeg}{\operatorname{y-deg}}

\theoremstyle{plain}

\newtheorem{theorem}{Theorem}[section]
\newtheorem{corollary}[theorem]{Corollary}
\newtheorem{lemma}[theorem]{Lemma}
\newtheorem{proposition}[theorem]{Proposition}

\theoremstyle{definition}
\newtheorem{definition}[theorem]{Definition}
\newtheorem{s}[theorem]{}
\newtheorem{remark}[theorem]{Remark}

\newtheorem{example}[theorem]{Example}
\newtheorem{examples}[theorem]{Examples}
\newtheorem*{example*}{\it Example}

\theoremstyle{remark}
\newtheorem*{claim*}{\it Claim}

\newtheorem*{case*}{\it Case}
\newtheorem*{note*}{\it Note}

\begin{document}
	\begin{frontmatter}
		\title{Asymptotic behaviour of bigraded components of local cohomology modules}
		\author[label1]{Rajsekhar Bhattacharyya}
		\ead{rbhattacharyya@gmail.com}
		\address[label1]{Dinabandhu Andrews College, Garia, Kolkata 700 084, India}		
		\author[label2]{Tony J. Puthenpurakal}
		\ead{tputhen@math.iitb.ac.in}
		\address[label2]{Department of Mathematics, Indian Institute of Technology Bombay, Powai, Mumbai, Maharashtra 400076, India}
		
		\author[label3]{Sudeshna Roy}
		\ead{sudeshna.roy@iitgn.ac.in}
		\address[label3]{Department of Mathematics, Indian Institute of Technology Gandhinagar, Palaj, Gandhinagar, Gujrat 382055, India}
		
		\author[label4]{Jyoti Singh}
		\ead{jyotijagrati@gmail.com}
		\address[label4]{Department of Mathematics, Visvesvaraya National Institute of Technology, South Ambazari Road, Nagpur, Maharashtra 440010, India}
		
\begin{abstract}
Let $C$ be a commutative Noetherian ring containing a field $K$ of characteristic zero. Let $R=C[X_1, \ldots, X_n, Y_1, \ldots, Y_m]$ be a polynomial ring over $C$ with $\bideg c=(0,0)$ for all $c \in C$, $\bideg X_i=(1,0)$ and $\bideg Y_j=(0,1)$ for $i=1, \ldots, n$ and $j=1, \ldots, m$. Let $I$ be a bihomogeneous ideal in $R$. In this article, we study asymptotic behaviour of bigraded pieces of the local cohomology module $H^i_I(R)$. Moreover, under the extra assumption that $C$ is regular, we investigate the asymptotic stability of invariants associated to its bigraded components. Consequently, we obtain certain properties of components of the bigraded local cohomology module $H^i_I(R)$, where $C=K$ is a field and $I$ is a binomial edge ideal.
     \end{abstract}
		
		\begin{keyword}
			local comohology \sep graded local cohomology\sep Weyl algebra, generalized Eulerian modules.
			\MSC[2010]Primary 13D45 \sep 14B15 \sep Secondary 13N10 \sep 32C36.
		\end{keyword}
		
	\end{frontmatter}

\section{Introduction}
Local cohomology was introduced by A. Grothendieck in the early 1960's seminars in Harvard in 1961. These modules play an important role in the study of algebraic and geometric properties of rings. Although several authors have studied these modules, their structure is still largely unknown. One drawback of local cohomology is that, in general, local cohomology modules are not necessarily finitely generated, and so can be somewhat complex to work with. Components of (multi)graded local cohomology modules supported on the irrelevant ideals are largely studied and have been applied drastically in both commutative algebra and algebraic geometry. In recent years, the local cohomology modules at monomial ideals gained considerable attention because of its  connection with the cohomology of coherent sheaves on a toric variety due to Cox \cite{Cox}. Besides, they are useful for constructing examples and non-examples. The main purpose of this article is to study the bigraded components of local cohomology modules supported on binomial ideals and investigate their asymptotic behaviors using the theory of $D$-modules. Note that many classical examples in algebraic geometry are associated to binomial ideals. 

Some authors have  used local cohomology as a vital tool in the study of binomial edge ideal \cite{Wil23_1} and \cite{Wil23_2}. Binomial edge ideals associate to a graph G a homogeneous ideal $\mathcal{J}_G$ in a polynomial ring over a field. Introduced by Herzog, Hibi, Hreinsdóttir, Kahle, and Rauh in \cite{HHHKR}, and independently by Ohtani in \cite{Oht11}, they found immediate application to the subject of conditional independence statements in statistics, and have since become widely studied objects in commutative algebra in their own right. There is a rich interplay between the algebraic properties of these ideals and the combinatorial properties of the corresponding graph. In \cite{Mon20}, \`Alvarez Montaner presented a Hochster type formula for the local cohomology modules of binomial edge ideals using the theory of local cohomology spectral sequences and provide a very explicit description of the graded pieces of these local cohomology modules. He presented a decomposition of the local cohomology modules of the generic initial ideal of a binomial edge ideal which is coarser than the original Hochster’s formula for monomial ideals. In this way he reduced the problem to a simple comparision of the graded pieces of the building blocks of these decompositions. In \cite{Wil23_1} and \cite{Wil23_2}, Williams studied the local cohomology modules $H^i_{\mathcal{J}_G}(R)$ and $H^i_{\mathfrak{m}}(R/\mathcal{J}_G)$ respectively, when $K$ is of prime characteristic $p > 0$. 

In this article, we primarily show that local cohomolgy modules supported on the following families of ideals, that have recently attracted the attention of several researchers, are generalized Eulerian bigraded modules (see Definition \ref{genEu_bi}): binomial edge ideals (see Examples \ref{eg}(i)) and certain determinantal ideals (see Examples \ref{eg}(ii)).

We are hopeful that the comprehensive study of bigraded components of local cohomology modules presented in this paper will be useful for the further study of binomial edge ideals and determinantal ideals. In this article, we do not use their underlying combinatorial structures, which could be a powerful tool for further study.

\s Let $S=\bigoplus_{u \in \N} S_u$ be a standard graded Noetherian ring, and let $L=\bigoplus_{u \in \Z} L_u$ be a finitely generated graded $S$-module. It is well known that for all $i \geq 0$,

(1) $H^i_{S_+}(L)_u$ is a finitely generated $S_0$-module for all $u \in \Z$,

(2) $H^i_{S_+}(L)_u=0$ for all $u \gg 0$,

\noindent
where $S_+=\bigoplus_{u>0} S_u$ is the irrelevant ideal of $S$, see \cite[Theorem 15.1.5]{BS}. In view of these facts, the second author investigated whether the components of local cohomology modules of $R=C[X_1, \ldots, X_n]$ with respect to arbitrary homogeneous ideals of $R$, have some predictable behavior. Many times it is seen that local cohomology modules of regular rings satisfy strong structural conditions. In \cite{TP2}, he naturally started with the case when $C$ is a regular ring containing a field of characteristic zero and presented a comprehensive study of graded components of local cohomology $H_I^i(R)$, where $I$ is an arbitrary homogeneous ideal in $R$, and also  found that they exhibit noticeable good behavior. Then along with the third author, in \cite{TPSR19} he moved to the case when $C=B^G$ is ring of invariants of a regular ring $B$ under the action of a finite group $G$. In \cite{TPSR}, they proved some results under more general setup when $C$ is any commutative Noetherian ring containing a field of characteristic zero and gave some examples to show the other results in \cite{TP2}.

The theory of $D$-modules which is a powerful tool connecting differential equations with algebraic geometry also bridged isolated fields like algebra, geometry, and mathematical physics. In \cite{Lyu1}, Lyubeznik, and later on several authors including the second author used the algebraic theory of D-modules very beautifully to study finiteness properties of (graded) local cohomology modules. In this article, we also use this theory to study properties of components of bigraded local cohomology modules, namely, vanishing, tameness and rigidity. We show that if $\TT$ is a bigraded Lyubeznik functor on ${}^* \Mod(R)$, then $\TT(R)$ is a bigraded generalized Eulerian module. Thus, our results on bigraded generalized Eulerian modules shed light on the properties of $\TT(R)$, especially local cohomology modules supported at arbitrary bihomogeneous ideals, not limited to monomial ones.

This article is structured as follows. In Section 2, we introduce and discuss certain
properties of bigraded generalized Eulerian modules and bigraded Lubeznik functors. In Section 3, we establish their  asymptotic vanishing behavior, Tameness and
rigidity properties. We provide a polynomial expression of $K$-vector space dimension of the bigraded components in Section 4. Section 5 is devoted to provide applications to local cohomology supported on bihomogeneous ideals.

\section{Notations}
Now we define some notations which will be used in rest of the paper.

\vspace{0.15cm}
For $(u_1,v_1), (u_2,v_2) \in \mathbb{Z}^2$, we define $(u_1,v_1) \leq (u_2,v_2)$ if $u_1 \leq u_2$ and $v_1 \leq v_2$.

\begin{definition}
A shaded region in the $uv$-plane is said to a \emph{block at  $(u_0, v_0) \in \mathbb{Z}^2$}, denoted by $\mathfrak{B}_{(u_0,v_0)}$, 
if it is the collection of points
\[\{(u, v) \in \mathbb{Z}^2 \mid u\geq u_0 \mbox{ if } u_0 \geq 0, u\leq u_0 \mbox{ if } u_0 \leq -n, v\geq v_0 \mbox{ if } v_0 \geq 0, \mbox{ and } v\leq v_0 \mbox{ if } v_0 \leq -m\},\]
\end{definition}

\noindent
Readers can refer to Remark \ref{block_fig} for a pictorial description. Sometimes, we write $\mathcal{B}$ instead of $\mathfrak{B}_{(u_0,v_0)}$ when the pair $(u_0, v_0)$ is arbitrary and does not carry much information. We further refer $(u_0,v_0)$ as the \emph{corner of the block $\mathfrak{B}_{(u_0,v_0)}$}.

\s\label{notations} {\bf Bigraded regions}. 

We shall represent the pieces of bigraded local cohomology modules to the co-ordinates in the $uv$-plane. We aim to capture their asymptotic behaviour in the regions defined below:  

\begin{minipage}{0.5\textheight}
\begin{enumerate}[\rm (i)]
	\item \emph{$\mathcal{NE}$-block or North East Block}: $\{(u,v)\in \Z^2$ with $(u,v) \geq (0,0)\}$.
	\item \emph{$\mathcal{NW^*}$-block  or North West Block}: $\{(u,v)\in \Z^2$ with $u \leq -n ~\text{and} ~ v\geq 0$\}.
	\item \emph{$\mathcal{S^*W^*}$-block  or South West Block}: $\{(u,v)\in \Z^2$ with $(u,v) \leq (-n,-m)\}$ 
	\item \emph{$\mathcal{S^*E}$-block  or South East Block}: $ \{(u,v)\in \Z^2$ with $u \geq 0 ~\text{and} ~ v \leq -m\}$
	\item \emph{$\mathcal{C}$-block or Complete Block}: $\{(u,v)\in \Z^2\}$ 
	\item \emph{$\mathcal{N}$-block or North Block}: $\{(u,v)\in \Z^2$ with $v \geq 0\}$.
	\item \emph{$\mathcal{W^*}$-block or Restricted West Block}: $\{(u,v)\in \Z^2$ with $u \leq -n\}$.
	\item  \emph{$\mathcal{S^*}$-block or Restricted South Block}: $\{(u,v)\in \Z^2$ with $v \leq -m\}$. 
	\item  \emph{$\mathcal{E}$-block  or East Block}: $\{(u,v)\in \Z^2$ with $u \geq 0\}$.	
\end{enumerate}
\end{minipage}
\begin{minipage}{0.225\textheight}
	\begin{center}
	\begin{tikzpicture}[scale=0.12]
	\draw[->] (-8.5,0)--(8.5,0) node[right]{$u$};
	\draw[->] (0,-8.5)--(0,8.5) node[above]{$v$};
	\draw[dashdotted, red] (-1,8.5)--(-1,-8.5) node[left]{$u=-n$};
	\draw[dashdotted, red] (-8.5,-1.5)--(8.5, -1.5) node[below]{$v=-m$};
	\draw[fill=cyan,fill opacity=0.35,draw=none] (0,8.5)--(0,0)--(8.5,0)--(8.5,8.5)--(0,8.5);
\node[red] at (0,0){\small{$\bullet$}};
	\node at (0,-11.25) {\footnotesize{\textit Figure: $\mathcal{NE}$}};
	\end{tikzpicture}
	\begin{tikzpicture}[scale=0.12]
	\draw[->] (-8.5,0)--(8.5,0) node[right]{$u$};
	\draw[->] (0,-8.5)--(0,8.5) node[above]{$v$};
	\draw[fill=cyan,fill opacity=0.35,draw=none] (-1,8.5)--(-1,0)--(-8.5,0)--(-8.5,8.5)--(-1,8.5);
	\draw[dashdotted, red] (-1,8.5)--(-1,-8.5);
\draw[dashdotted, red] (-8.5,-1.5)--(8.5, -1.5);
\node[red] at (-1,0){\small{$\bullet$}};
	\node at (0,-11) {\footnotesize{\textit Figure: $\mathcal{NW^*}$}};
	\end{tikzpicture}
\end{center}
\end{minipage}
\begin{center}
	\begin{tikzpicture}[scale=0.12]
	\draw[->] (-8.5,0)--(8.5,0) node[right]{$u$};
	\draw[->] (0,-8.5)--(0,8.5) node[above]{$v$};
	\draw[fill=cyan,fill opacity=0.35,draw=none] (-1,-8.5)--(-1,-1.5)--(-8.5,-1.5)--(-8.5,-8.5)--(-1,-8.5);
	\draw[dashdotted, red] (-1,8.5)--(-1,-8.5);
\draw[dashdotted, red] (-8.5,-1.5)--(8.5, -1.5);
\node[red] at (-1,-1.5){\small{$\bullet$}};
	\node at (0,-11) {\footnotesize{\textit Figure: $\mathcal{S^*W^*}$}};
	\end{tikzpicture}
	\hspace{0.25cm}
	\begin{tikzpicture}[scale=0.12]
	\draw[->] (-8.5,0)--(8.5,0) node[right]{$u$};
	\draw[->] (0,-8.5)--(0,8.5) node[above]{$v$};
	\draw[fill=cyan,fill opacity=0.35,draw=none] (0,-8.5)--(0,-1.5)--(8.5,-1.5)--(8.5,-8.5)--(0,-8.5);
	\draw[dashdotted, red] (-1,8.5)--(-1,-8.5);
\draw[dashdotted, red] (-8.5,-1.5)--(8.5, -1.5);
\node[red] at (0,-1.5){\small{$\bullet$}};
	\node at (0,-11) {\footnotesize{\textit Figure: $\mathcal{S^*E}$}};
	\end{tikzpicture}	
	\hspace{0.25cm}
			\begin{tikzpicture}[scale=0.12]
	\draw[->] (-8.5,0)--(8.5,0) node[right]{$u$};
	\draw[->] (0,-8.5)--(0,8.5) node[above]{$v$};
\draw[dashdotted, red] (-1,8.5)--(-1,-8.5);
\draw[dashdotted, red] (-8.5,-1.5)--(8.5, -1.5);
	\draw[fill=cyan,fill opacity=0.35,draw=none] (-8.5,8)--(8.5,8.5)--(8.5,0)--(-8.5,0);
	\node at (0,-11) {\footnotesize{\textit Figure: $\mathcal{N}$}};
	\end{tikzpicture}
	\hspace{0.25cm}
\begin{tikzpicture}[scale=0.12]
\draw[->] (-8.5,0)--(8.5,0) node[right]{$u$};
\draw[->] (0,-8.5)--(0,8.5) node[above]{$v$};
\draw[dashdotted, red] (-1,8.5)--(-1,-8.5);
\draw[dashdotted, red] (-8.5,-1.5)--(8.5, -1.5);
\draw[fill=cyan,fill opacity=0.35,draw=none] (-8.5,8.5)--(-1,8.5)--(-1,-8.5)--(-8.5,-8.5);
\node at (0,-11) {\footnotesize{\textit Figure: $\mathcal{W^*}$}};
\end{tikzpicture}	
\vspace{0.25cm}
\begin{tikzpicture}[scale=0.12]
\draw[->] (-8.5,0)--(8.5,0) node[right]{$u$};
\draw[->] (0,-8.5)--(0,8.5) node[above]{$v$};
\draw[dashdotted, red] (-1,8.5)--(-1,-8.5);
\draw[dashdotted, red] (-8.5,-1.5)--(8.5, -1.5);
\draw[fill=cyan,fill opacity=0.35,draw=none] (-8.5,-1.5)--(8.5,-1.5)--(8.5,-8.5)--(-8.5,-8.5);
\node at (0,-11) {\footnotesize{\textit Figure: $\mathcal{S^*}$}};
\end{tikzpicture}				
\vspace{0.25cm}
\begin{tikzpicture}[scale=0.12]
\draw[->] (-8.5,0)--(8.5,0) node[right]{$u$};
\draw[->] (0,-8.5)--(0,8.5) node[above]{$v$};
\draw[dashdotted, red] (-1,8.5)--(-1,-8.5);
\draw[dashdotted, red] (-8.5,-1.5)--(8.5, -1.5);
\draw[fill=cyan,fill opacity=0.35,draw=none] (0,8)--(8,8)--(8,-8.5)--(0,-8.5);
\node at (0,-11) {\footnotesize{\textit Figure: $\mathcal{E}$}};
\end{tikzpicture}			
\end{center}

\begin{remark}\label{block_fig}
Note that $(i), (ii), (iii),$ and $(iv)$ are associated to $\mathfrak{B}_{(0,0)}, \mathfrak{B}_{(-n,0)}, \mathfrak{B}_{(-n,-m)}$ and $\mathfrak{B}_{(0,-n)}$, respectively.	
\end{remark}	

For our study, we also consider the truncated regions: $(i)~ \mathrm{Trun}(\mathcal{N}):=\mathcal{N} \backslash \big(\mathcal{NE}\cup \mathcal{NW^*}\big),~ (ii)~ \mathrm{Trun}(\mathcal{W^*}):=\mathcal{W^*}\backslash \big(\mathcal{NW^*}\cup \mathcal{S^*W^*}\big),~ (iii)~ \mathrm{Trun}(\mathcal{S^*}):=\mathcal{S^*}\backslash \big(\mathcal{S^*W^*}\cup \mathcal{S^*E}\big)$, and $(iv)~ \mathrm{Trun}(\mathcal{E}):=\mathcal{E}\backslash\big(\mathcal{S^*E}\cup \mathcal{NE}\big) \big\}$
with their \emph{pair of corners} $\{(0,0), (-n,0)\}, \{(-n,0),(-n,-m)\}, \{(-n,-m),(-m,0)\}$ and $\{(-m,0),(0,0)\}$, respectively. We pictorially present these regions below and highlight their corners by red dots.

\begin{center}
	\begin{tikzpicture}[scale=0.12]
\draw[->] (-7.5,0)--(8.5,0) node[right]{$u$};
\draw[->] (0,-7.5)--(0,8.5) node[above]{$v$};
\draw[red,dashdotted] (8,-3)--(-7.5,-3)node[left]{$v=-m$};
\draw[red, dashdotted] (-2,-7.5)--(-2,8) node[left]{$u=-n$};
\node[red] at (0,0){\small{$\bullet$}};
\node[red] at (-2,0){\small{$\bullet$}};
\draw[fill=cyan,fill opacity=0.35,draw=none] (-2,8)--(0,8)--(0,0)--(-2,0);
\node at (0,-11) {\footnotesize{\textit Figure: $\mathrm{Trun}(\mathcal{N})$}};
\end{tikzpicture}					
\vspace{0.15cm}
\begin{tikzpicture}[scale=0.12]
\draw[->] (-7.5,0)--(8.5,0) node[right]{$u$};
\draw[->] (0,-7.5)--(0,8.5) node[above]{$v$};
\draw[red,dashdotted] (8,-3)--(-7.5,-3);
\draw[red, dashdotted] (-2,-7.5)--(-2,8); 
\node[red] at (-2,0){\small{$\bullet$}};
\node[red] at (-2,-3){\small{$\bullet$}};
\draw[fill=cyan,fill opacity=0.35,draw=none] (-2,0)--(-8,0)--(-8,-3)--(-2,-3);
\node at (0,-11) {\footnotesize{\textit Figure: $\mathrm{Trun}(\mathcal{W^*})$}};
\end{tikzpicture}	
\hspace{0.15cm}
\begin{tikzpicture}[scale=0.12]
\draw[->] (-7.5,0)--(8.5,0) node[right]{$u$};
\draw[->] (0,-7.5)--(0,8.5) node[above]{$v$};
\draw[red,dashdotted] (8,-3)--(-7.5,-3);
\draw[red, dashdotted] (-2,-7.5)--(-2,8);
\node[red] at (-2,-3){\small{$\bullet$}};
\node[red] at (0,-3){\small{$\bullet$}};
\draw[fill=cyan,fill opacity=0.35,draw=none] (-2,-3)--(0,-3)--(0,-8)--(-2,-8);
\node at (0,-11) {\footnotesize{\textit Figure: $\mathrm{Trun}(\mathcal{S^*})$}};
\end{tikzpicture}	
\hspace{0.15cm}
\begin{tikzpicture}[scale=0.12]
\draw[->] (-7.5,0)--(8.5,0) node[right]{$u$};
\draw[->] (0,-7.5)--(0,8.5) node[above]{$v$};
\draw[red,dashdotted] (8,-3)--(-7.5,-3);
\draw[red, dashdotted] (-2,-7.5)--(-2,8); 
\node[red] at (0,0){\small{$\bullet$}};
\node[red] at (0,-3){\small{$\bullet$}};
\draw[fill=cyan,fill opacity=0.35,draw=none] (0,0)--(0,-3)--(8,-3)--(8,0);
\node at (0,-11) {\footnotesize{\textit Figure: $\mathrm{Trun}(\mathcal{E})$}};
\end{tikzpicture}
\end{center}	

\section{Bigraded generalized Eulerian module}

\s \label{sa} {\bf Setup}. Suppose that $C$ is a commutative Noetherian ring containing a field $K$ of characteristic zero. Let $R=C[X_1, \ldots, X_n, Y_1, \ldots, Y_m]$ be a polynomial ring over $C$. Let $A_{m,n}(C)=R \langle \partial_1, \ldots, \partial_n, \delta_1, \ldots, \delta_n\rangle$ be the bigraded $(n,m)^{th}$-Weyl algebra over $C$, where $\partial_i=\partial/\partial X_i$ and  $\delta_j=\partial/\partial Y_j$ for $i=1, \ldots, m$ and $j=1, \ldots, n$. Consider both $R$ and $A_{m,n}(C)$ as standard $\Z^2$-graded with $\bideg (c)=(0,0)$ for all $c \in C$, $\bideg X_i=(1,0), \bideg Y_j=(0,1), \bideg \partial_i=-(1,0)$ and $\bideg \delta_i=-(0,1)$ for $i=1, \ldots, m$ and $j=1, \ldots, n$. To distinguish the $(n+m)^{th}$-Weyl algebra $A_{n+m}(C)$ over $C$ with its \emph{standard $\mathbb{Z}$-graded structure}, here we denote it by $A_{m,n}(C)$. Let $I$ be a bihomogeneous ideal in $R$. Then the local cohomology module $M:=H^i_I(R)=\bigoplus_{(u,v)} M_{(u,v)}$ has an induced bigraded structure. 

\newpage
\begin{examples}\label{eg}

\noindent
\begin{enumerate}[\rm (i)]
\item Let $R=K[X_1, \ldots, X_m, Y_1, \ldots, Y_m]$ be a polynomial ring over a field $K$. Let $G$ be a graph on $m$ vertices. By $E(G)$, we denote the set of edges $\{i, j\}$ of $G$. We define the \emph{binomial edge ideal of $G$} as
\[\mathcal{J}_G := (X_iY_j-X_jY_i \mid \{i, j\} \in E(G))\]
in $R$. Let $\mathfrak{m}=(X_1, \ldots, X_m, Y_1, \ldots, Y_m)$ denote the unique homogeneous maximal ideal in $R$. Recall that binomial edge ideals are radical. Note that $\mathcal{J}_G$ is a bihomogeneous ideal under the assumptions in \ref{sa}. The graded components of $H^i_{\mathfrak{m}}(R/\mathcal{J}_G)$ have been analyzed in \cite{Mon20}, when $R$ is standard $\mathbb{N}^m$-graded with $\deg X_j=\deg Y_j=e_j\in \mathbb{N}^m$, the $j^{th}$-canonical unit vector, for $j=1, \ldots, m$. 
 
 \vspace{0.15cm}
\item There are different kind of determinantal ideals (mostly with $2 \times 2$-minors) which are bihomogeneous under \ref{sa}. One such ideal is the double determinantal ideal, introduced by L. Li in \cite{L23}. Fix $r,m,n \geq 1$. Let ${\bf X} = \{{\bf X}^{(k)}\}_{k=1}^r$, where ${\bf X}^{(k)}=\big(X^{(k)}_{ij}\big)_{m \times n}$ is an $m\times n$ matrix consisting of independent variables. Let $R=K[\bf X]$ be a polynomial ring over a perfect field $K$. Let $H$ be the horizontal concatenation of these matrices, i.e., the $m \times rn$ matrix $H=\big({\bf X}^{(1)} \cdots {\bf X}^{(r)}\big)$ and let $V$ be their vertical concatenation, i.e., the $rm \times n$ matrix
{\small\[V=\left( {\begin{array}{c}
	{\bf X}^{(1)} \\
	\vdots\\
	{\bf X}^{(k)}
	\end{array} }\right).\]}

\noindent
The ideal $J$ generated by the $s$-minors of $H$ together with the $t$-minors of $V$ is called a \emph{double determinantal ideal}. Pick a nonempty proper subset $U\subset \{1, \ldots, m\}$. For $i \in U$, set $\bideg X^{(k)}_{ij}=(1,0)$  for all $j=1, \ldots, n$ and all $k=1, \ldots, r$. Similarly, for $i \notin U$, set $\bideg X^{(k)}_{ij}=(0,1)$ for all $j=1, \ldots, n$ and all $k=1, \ldots, r$. This will induce a standard bigraded structure on $R$. Notice $J$ is a bihomogeneous ideal in $R$ if $s=t=2$.
\end{enumerate} 
\end{examples}

\s We define the \emph{bigraded Euler operator} on $A_{n,m}(C)$ as $(\mathcal{E}^X_n, \mathcal{E}^Y_m)$ by setting
\[\mathcal{E}^X_n= \sum_{i=1}^n X_i \partial_i \quad \mbox{ and } \quad \mathcal{E}^Y_m= \sum_{j=1}^m Y_j \delta_j.\]

\begin{remark}\label{24hlc}
	Let \[
	A=
	{\begin{array}{cc}
		\underbrace{\left[ {\begin{array}{cccc}
		1 & 1 & \cdots & 1  \\
		0 & 0 & \cdots & 0 \\
		\end{array} }\right.}_n &
	\underbrace{\left.\begin{array}{cccc}
	0 & \cdots & 0 & 0 \\
	1 & \cdots & 1 & 1\\
		\end{array}\right]}_m  \end{array}} 
	\]
	be a $2 \times (n+m)$ integer matrix of rank $2$. Denote the columns of $A$ by $a_1, \ldots, a_{n+m}$. Then following the notations in \cite[page 248]{24Hrs} we get that $\bideg X_i=a_i, \bideg \partial_i=-a_i, \bideg Y_j=a_{n+j}$ and $\bideg \delta_j=-a_{n+j}$ for $i=1, \ldots, n$ and $j=1, \ldots, m$, which matches with our assumptions in \ref{sa}. Besides, $\e^X_n$ and $\e^Y_m$ are essentially the first and second Euler operators $E_1$ and $E_2$ in \cite{24Hrs}, respectively when $C=K$ is a field.
\end{remark}
We call a bigraded $A_{n,m}(C)$-module $E$ is {\it Eulerian} if for every bihomogeneous element $e$ in $E$ with $\bideg e=(u,v)$ \[(\e^X_n-u) \cdot e=0 \quad \mbox{and} \quad (\e^Y_m-u) \cdot e=0.\]  

We further define the following extended class of bigraded $A_{n,m}(C)$-modules. 

\begin{definition}\label{genEu_bi}
A bigraded $A_{n,m}(C)$-module $E$ is said to be {\it generalized Eulerian} if for each bihomogeneous element $e$ of $E$ with $\bideg e=(u,v)$, there exists an integer $a>0$ (depending on $e$) such that
\[(\mathcal{E}^X_n-u)^a \cdot e=0 \quad \mbox{ and } \quad (\mathcal{E}^Y_m-v)^a \cdot e=0.\]
\end{definition}
Note that bigraded generalized Eulerian $A_{1,1}(C)$-modules have already been introduced in \cite{TPSR24} to study ideals generated by monomials with coefficients (not necessarily units) in $C$.

\begin{remark}\label{Rel-bigrd-grd-genEur}
Consider both $R$ and $A_{n,m}(C)$ as standard graded with $\deg X_i=\deg Y_j=1$ and $\deg \partial_i=\deg \delta_j=-1$ for all $i=1, \ldots, m$ and $j=1, \ldots, n$. To distinguish, we denote $A_{n,m}(C)$ by $A_{m+n}(C)$ in this case. Given a bigraded $A_{n,m}(C)$-module $E$, we further consider the $A_{m+n}(C)$-graded module $\mbox{Tor}(E)$ defined by  $\mbox{Tor}(E)_r=\bigoplus_{u+v=r}M_{(u,v)}$ for every $r \in \mathbb{Z}$. Pick $e \in E_{(u,v)}$ such that $(\mathcal{E}^X_n-u)^a \cdot e=0$ and $(\mathcal{E}^Y_m-v)^a \cdot e=0$ for some integer $a>0$. Notice $\e_n^X-u$ and $\e_m^Y-v$ commutes with each other in $A_{n,m}(C)$. Hence \[\left(\e_n^X+\e_m^Y-(u+v)\right)^{2a} \cdot e=\left((\e_n^X-u)+(\e_m^Y-v)\right)^{2a} \cdot e=\sum_{i=0}^{2a} \binom{2a}{i}(\e_n^X-u)^{2a-i}(\e_m^Y-v)^i=0.\] This shows that if $E$ is a bigraded generalized Eulerian $A_{n,m}(C)$-module then $\mbox{Tot}(E)$ is a graded generalized Eulerian $A_{m+n}(C)$-module. 
\end{remark}

\begin{remark}\label{comp-grd-genEur}
Suppose that $M=\bigoplus_{(u,v) \in \mathbb{Z}^2}M_{(u,v)}$ is a bigraded $A_{n,m}(C)$-module. Set $R^X:=C[X_1, \ldots, X_n]$ and $R^Y:=C[Y_1,\ldots, Y_m]$. Let $R^X, R^Y$ be standard graded with $\deg X_i=1=\deg Y_j$, $\deg c=0$ for all $c \in C$. Fix $(u_0,v_0)\in \Z^2$. Write $M_{u_0, *}:=\bigoplus_{v \in \Z} M_{(u_0,v)}$ and $M_{*,v_0}:=\bigoplus_{u \in \Z} M_{(u,v_0)}$. Clearly $M_{u_0, *}$ and $M_{*,v_0}$ are graded $R^Y$ and $R^X$-modules, respectively. Suppose that $M$ is a bigraded generalized Eulerain $A_{n,m}(C)$-module. Then by definition for each $e_1 \in M_{u_0,v}$ (resp. $e_2 \in M_{u,v_0}$) we have $(\mathcal{E}^Y_m-v)^a \cdot e_1=0$ (resp. $(\mathcal{E}^X_n-u)^a \cdot e_2=0$) for some $a \geq 1$. Thus $M_{*,v_0}$ (resp. $M_{u_0,*}$) is a graded generalized Eulerian $A^X_n(C)=R^X\langle \partial_1, \ldots, \partial_n\rangle$-module (resp. $A_m^Y(C)=R^Y\langle \delta_1, \ldots, \delta_m\rangle$-module).
\end{remark}

The following result, asserting that the class of bigraded generalized Eulerian modules is closed under extensions, is crucial to this study. As it can be proved along the same lines as \cite[Proposition 2.1]{TP1}, we omit the proof.

\begin{proposition}\label{exact}
	Let $0 \rt M_1 \xrightarrow{\alpha_1} M_2 \xrightarrow{\alpha_2} M_3 \rt 0$ be a short exact sequence of bigraded $A_{n,m}(C)$-modules (all maps are bihomogeneous).
	Then the following are equivalent:
	\begin{enumerate}[\rm (1)]
		\item
		$M_2$ is bigraded generalized Eulerian.
		\item
		$M_1$ and $M_3$ are bigraded generalized Eulerian.
	\end{enumerate}
\end{proposition}

Following the same methods as in \cite[Propositions 3.2 and 3.5]{TP1} and \cite[Propositions 5.3 and 5.5]{TPJS} (see also \cite[Propositions 3.4, 3.5, 3.6, and 3.7]{TPSR}), one can establish the following results.

\begin{proposition}\label{bigrd-genEur-XY}
Let $M$ be a bigraded generalized Eulerian $A_{n,m}(C)$-module. Then for $i=0,1$, the $A_{n-1,m}(C)$-module $H_i(X_n; M)$ and the $A_{n,m-1}(C)$-module $H_i(Y_m; M)$ are bigraded generalized Eulerain. 
\end{proposition}

\begin{proposition}\label{bigrd-genEur-partial}
Let $M$ be a bigraded generalized Eulerian $A_{n,m}(C)$-module. Then for $i=0,1$, the $A_{n-1,m}(C)$-module $H_i(\partial_n; M)(-1,0)$ and the $A_{n,m-1}(C)$-module $H_i(\delta_m; M)(0,-1)$ are bigraded generalized Eulerain . 
\end{proposition}
\begin{proposition}\label{bigrd-genEur-partial-con}
Let $M$ be a bigraded generalized Eulerian $A_{n,m}(C)$-module. For $n=1$, $H_i(\partial_1; M)_{(u,v)} \neq 0$ only if $u=-1$ and for $m=1$, $H_i(\delta_1; M)_{(u,v)} \neq 0$ only if $v=-1$ where $i=0,1$. 
\end{proposition}
\begin{proposition}\label{bigrd-genEur-XY-con}
Let $M$ be a bigraded generalized Eulerian $A_{n,m}(C)$-module. For $n=1$, $H_i(X_1; M)_{(u,v)} \neq 0$ only if $u=0$ and for $m=1$, $H_i(Y_1; M)_{(u,v)} \neq 0$ only if $v=0$ where $i=0,1$. 
\end{proposition}

\begin{remark}
Suppose $n>1$. Then using similar arguments as in \cite[Theorem 5.4]{TPJS}, we can prove that for each $j \geq 0$, the de Rahm homology module $H_j(X_i, X_{i+1}, \ldots, X_n; M)$ is a bigraded generalized Eulerian $A_{i-1, m}(C)$-module. In particular, $N^{(j)}:= H_j(X_2, X_3, \ldots, X_n; M)$ is a bigraded generalized Eulerian $A_{1, m}=C\langle X_1, Y_1, \ldots, Y_m, \partial_1, \delta_1, \ldots, \delta_m\rangle$-module. Fix $j \geq 1$. For $i=0,1$, by Proposition \ref{bigrd-genEur-XY-con},
\[H_i(X_1; N^{(j)})_{(0,v)} \neq 0 \quad \mbox{ and } \quad H_i(X_1; N^{(j)})_{(u,v)} = 0 \quad \mbox{ for all } u \neq 0.\] 

\noindent
\begin{minipage}{0.625\textheight}
Set 
$\underline{X}:=X_1, X_2, \ldots, X_n$. Then from the exact sequence 
\[0 \to H_0(X_1; N^{(j)}) \to H_j(\underline{X}; M) \to H_1(X_1, N^{(j-1)}) \to 0 \]
of graded $R=C[X_1, Y_1, \ldots, Y_m]$-modules (see \cite[Lemma 2.3]{TPJS}), we get $H_j(\underline{X}; M)_{(u,v)} \neq 0$ only if $u=0$. Thus, one may consider $W^{(j)}:=H_j(\underline{X}; M)=\bigoplus_{v \in \Z} \big(W^{(j)}\big)_{*,v}$ as a graded $A_m(C)=C\langle Y_1, \ldots, Y_m,$ $ \delta_1, \ldots, \delta_m\rangle$-module with $\big(W^{(j)}\big)_{*,v}=H_j(\underline{X}; M)_{(0,v)}$. The nonzero graded pieces are shown in Figure $(i)$. Given Proposition \ref{bigrd-genEur-partial-con}, the preceding arguments extend to $H_j(\underline{\partial}; M)$, where
\end{minipage}
\begin{minipage}{0.115\textheight}
\begin{center}
	\begin{tikzpicture}[scale=0.115]
	\draw[->] (-7.5,0)--(8.5,0) node[below]{$u$};
	\draw[->] (0,-7.5)--(0,8.5) node[above]{$v$};
	\draw[-, line width=0.25mm, blue] (0,-7.5)--(0,8);
	\node at (-0.25,-10) {\text{\small Figure} $(i)$};
	\node at (-6.25,5) {\scriptsize $H_j(\underline{X}; M)$};
	\end{tikzpicture}
\end{center}
\end{minipage}

\noindent
$\underline{\partial}=\partial_1, \ldots, \partial_n$. 
Thus
$W^{(j)}:=H_j(\underline{\partial}; M)=\bigoplus_{v \in \Z} \big(W^{(j)}\big)_{*,v}$ is a graded $A_m(C)$-module with $\big(W^{(j)}\big)_{*,v}=H_j(\underline{\partial}; M)_{(-1,v)}$.
\end{remark}

\section{Bigraded Lyubeznik functors}\label{mul-Lyubeznik}

{\normalfont Let $\TT$ be a graded Lyubeznik functor on ${}^*\Mod(R)$. To prove $\TT(R)$ is graded generalized Eulerian, in \cite{TP2}, the second author considered a general setup which enables us to prove that $\TT(R)$ is biigraded generalized Eulerian when $\TT$ is a bigraded Lyubeznik functor on ${}^*\Mod(R)$. We first recall the setup.
	
	\vspace{0.15cm}
	\noindent
	{\it Setup}: Let $C$ be a commutative Noetherian ring containing a field $K$ of characteristic zero. Assume that there is a ring $\Lambda$ (not necessarily commutative) containing $C$ such that $K \subseteq Z(\Lambda)$, the center of $\Lambda$. Furthermore assume that
	\begin{enumerate}[\rm (i)]
		\item $\Lambda$ is free both as a left $C$-module and as a right $C$-module.
		\item $N = C$ is a left $\Lambda$-module such that if we restrict the $\Lambda$ action on $N$ to $C$ we get the usual
		action of $C$ on $N$.
	\end{enumerate}
	Set $R = C[X_1,\ldots, X_n]$ and $\Gamma = \Lambda[X_1,\ldots, X_n]$. Consider $R$ as graded with $\deg C = 0$ and $\deg X_i = 1$ and $\Gamma$ as graded with $\deg \Lambda = 0$ and $\deg X_i = 1$ for all $i$. Also assume that for each homogeneous ideal $I$ of $R$ and a graded $\Gamma$-module $L$, the set $H^0_I(L)$ is a graded $\Gamma$-submodule of $L$. Let $\D_n:=A_n(\Lambda)$ be the $n^{th}$-Weyl algebra over $\Lambda$. We consider $\D_n$ graded by assigning	$\deg \Lambda = 0$, $\deg X_i = 1$ and $\deg \partial_i = -1$. 
	
	Under the above setup the second author proved the following. 
	\begin{theorem}\cite[Theorem 3.6, Lemma 3.7]{TP2}\label{gen-eul}
		Let $\TT$ be a graded Lyubeznik functor on ${}^* Mod(R)$. Then $\TT(R)$ is a graded generalized Eulerian $\D_n$-module.
	\end{theorem}
	
For our subsequent study, we generalize the preceding result to the case $\Lambda = C$ and outline a proof.
\begin{theorem}\label{Mul-LC}
Let $C$ be a Noetherian ring containing a field of characteristic zero and $R=C[X_1, \ldots, X_n, Y_1, \ldots, Y_m]$ be a polynomial ring over $C$. Let $A_{m,n}(C)=R \langle \partial_1, \ldots, \partial_n, \delta_1, \ldots, \delta_m\rangle$ denote the bigraded $(n,m)^{th}$-Weyl algebra over $C$, where $\partial_i=\partial/\partial X_i$ and $\delta_j=\partial/\partial Y_j$. Consider both $R$ and $A_{m,n}(C)$ as standard $\Z^2$-graded with $\bideg (c)=(0,0)$ for all $c \in C$, $\bideg X_i=(1,0), \bideg Y_j=(0,1), \bideg \partial_i=-(1,0)$ and $\bideg \delta_i=-(0,1)$ for $i=1, \ldots, n$ and $j=1, \ldots, m$. Let $\mathcal{T}$ be a multigraded Lyubeznik functor on ${}^*\Mod(R)$. Then $\TT(R)$ is a multigraded generalized Eulerian $A_{n,m}(C)$-module.
\end{theorem}

	\begin{proof}	
		Let $I$ be a bihomogeneous ideal in $R$ and $M$ a bigraded $A_{n,m}(C)$-module.
		
		\begin{claim*}
			$H_I^0(M)$ is a bigraded $A_{n,m}(C)$-submodule of $M$. 	
		\end{claim*} 
		
		Let $y \in H_I^0(M)$. Then $I^s y=0$ for some $s \geq 1$. Take $f \in I^{s+1}$. For any $g \in R$,
		\[\left(\partial_i f\right)(g)=\partial_i(fg)=f \partial_i(g)+g \partial_i(f) =\left(f \partial_i+\partial_i(f)\right)(g)\]
		and hence $f\partial_i= \partial_i f-\partial_i(f)$. As $\partial_i(f) \in I^s$ so $f \partial_i y=0$. Thus $\partial_i y \in H_I^0(M)$ for $i=1,\ldots, n$. Similarly, we can show that  $\delta_jy \in H_I^0(M)$ for $j=1,\ldots, m$. Note that $A_{n,m}(C)$ is a free $R$-module with a set of generators 
		\begin{equation*}
		\{\partial_1^{\alpha_1} \cdots \partial_n^{\alpha_n} \delta_1^{\beta_1}\cdots \delta_m^{\beta_m} \mid \alpha_{i}, \beta_j \geq 0 \mbox{ for } i=1, \ldots, n \mbox { and } j=1, \ldots, m\}.
		\end{equation*}
		The claim follows.
		
Using the same method as in \cite[3.10]{TP2}, one can easily prove that $\mathcal{T}(R)$ is a $\Z^2$-graded $A_{n,m}(C)$-module. We set $\mathcal{T}(R):=M= \bigoplus_{(u,v)\in \Z^2}M_{(u,v)}$. 
		
Set $S^X=R^X[Y_1, \ldots, Y_m]$ and $S^Y=R^Y[X_1, \ldots, X_n]$. Clearly, $R=S^X=S^Y$. Consider $S^X$ as graded with $\deg z=0$ for all $z \in R^X$ and $\deg Y_j=1$ for $j=1,\ldots, m$. Then $\mathcal{T}$ is a graded {\it Lyubeznik functor} on ${}^*\Mod(S^X)$. In view of \cite[Setup 3.1]{TP2} and by \cite[Theorem 3.6]{TP2} it follows that $\mathcal{T}(R)=\mathcal{T}(S^X)=\bigoplus_{v \in \Z} M_{*,v}$ is a graded generalized Eulerian $\mathcal{D}_m^X:=A_m\big(A_n(C)\big)=A_m\big(R^X\langle\partial_1, \ldots, \partial_n\rangle\big)$-module. Similarly, $\mathcal{T}(R)=\bigoplus_{u \in \Z} M_{u,*}$ is graded generalized Eulerian $\mathcal{D}_n^Y:=A_n\big(A_m(C)\big)=A_n\big(R^Y\langle\delta_1, \ldots, \delta_m\rangle\big)$-module if we consider $S^Y$ as graded with $\deg z=0$ for all $z \in R^Y$ and $\deg X_i=1$ for $i=1, \ldots, n$. Given $(u_0, v_0) \in \Z^2$, note that $M_{(u_0, *)}=\bigoplus_{v \in \Z} M_{(u_0,v)}$ and $M_{*,v_0}=\bigoplus_{u \in \Z} M_{(u,v_0)}$. Thus for all $e \in M_{(u_0,v_0)}$ there exists $a,b \in \Z_{>0}$ such that $(\mathcal{E}^Y_m-v_0)^a \cdot e=0$ and $(\mathcal{E}^X_n-u_0)^b \cdot e=0$. By definition it follows that $\mathcal{T}(R)$ is a bigraded generalized Eulerian $A_{n,m}(C)$-module.
	\end{proof}

	\section{Properties of components of bigraded generalized Eulerian modules}

	The purpose of this section is to discuss some properties of components of bigraded local cohomology modules. In light of Theorem \ref{Mul-LC}, more generally we study components of bigraded generalized Eulerian modules under the following setup.

	\s {\bf Standard assumptions:}\label{setup_bigraded}
Let $C$ be any Noetherian ring containing a field $K$ of characteristic zero and $R=C[X_1, \ldots, X_n, Y_1, \ldots, Y_n]$ be a polynomial ring over $C$. Let $A_{m,n}(C)=R \langle \partial_1, \ldots, \partial_n, \delta_1, \ldots, \delta_n\rangle$ denote the bigraded $(n,m)^{th}$-Weyl algebra over $C$, where $\partial_i=\partial/\partial X_i$ and $\delta_j=\partial/\partial Y_j$. Consider both $R$ and $A_{m,n}(C)$ as standard $\Z^2$-graded with $\bideg (c)=(0,0)$ for all $c \in C$, $\bideg X_i=(1,0), \bideg Y_j=(0,1), \bideg \partial_i=-(1,0)$ and $\bideg \delta_i=-(0,1)$ for $i=1, \ldots, m$ and $j=1, \ldots, n$. Let $M= \bigoplus_{(u,v)\in \Z^2}M_{(u,v)}$ be a bigraded generalized Eulerian $A_{m,n}(C)$-module. 

\vspace{0.15cm}
We restrict our attention to vanishing, tameness, and rigidity problems. Since these properties are known for graded case (see \cite{TP2}, \cite{TPSR}), we can use Remarks \ref{Rel-bigrd-grd-genEur} and \ref{comp-grd-genEur}.

\s \label{van_tem}{\bf Vanishing property.}\\
We  prove that vanishing of almost all bigraded components of countably generated bigraded module $M$ implies vanishing of $M$. We first prove the following lemma.
\begin{lemma}\label{lem1}
Let $M=\bigoplus_{(u,v)\in \Z^2}M_{(u,v)}$ be a countably generated multigraded  $R=\bigoplus_{(u,v)\in \N^2}R_{(u,v)}$-module. Then $M_{(u,v)}$ is a countably generated $R_{(0,0)}$-module for every $(u,v)\in \Z^2$.
\end{lemma}
\begin{proof}
Suppose that $\{m_1, m_2, \ldots\}$ is a countable generating set of $M$ as an $R$-module with $\bideg m_i= (u_i,v_i) \in \Z^2$. We have an exact sequence of bigraded $R$-modules
\[\bigoplus_{i=1}^{\infty} R\big(-(u_i,v_i)\big) \to M \to 0,\] which induces an exact sequence of $R_{(0,0)}$-modules
\[\bigoplus_{i=1}^{\infty} R_{(u-u_i,v-v_i)} \to M_{(u,v)} \to 0\] for all $(u,v) \in \Z^2$. Since each $R_{(u-u_i,v-v_i)}$ is a finitely generated $R_{(0,0)}$-module, it follows that $M_{(u,v)}$ is a countably generated $R_{(0,0)}$-module.
\end{proof}

\begin{theorem}[Vanishing]\label{lem2}{\rm(}With hypotheses as in \ref{setup_bigraded}{\rm)}
Further assume that $K$ is a uncountable field and $M= \bigoplus_{u,v\in \Z}M_{u,v}$ is a countably generated bigraded $R$-module. Suppose $M_{(u,v)}= 0$ for every $(u,v) \notin S$, where $S$ is the shaded region in Figure $(i)$, then $M=0$.
\end{theorem}

\begin{center}
	\begin{tikzpicture}[scale=0.175]
	\draw[->] (-7.5,0)--(8.5,0) node[right]{$u$};
	\draw[red, densely dashdotted] (-7.5,-5.5)--(8.5,-5.5) node[right]{$v=v_0$};
	\draw[->] (0,-7.5)--(0,8.5) node[above]{$v$};
	\draw[line width=0.25mm,blue](-1.5,-7.25)-- (-1.5,7.25)node[left]{$u=-b_1$};
	\draw[blue](2,-7.25)-- (2,7.25) node[right]{$u=a_1$};
	\draw[blue](-7.5,-1.5)-- (7.5,-1.5) node[below]{$v=-b_2$};
	\draw[blue]  (-7.5,2)-- (7.5,2) node[above]{$v=a_2$};
	\draw[fill=cyan,fill opacity=0.35,draw=none] (-1.5,7.25)--(-1.5,-7.5)-- (2,-7.5) --(2,7.25);
	\draw[fill=cyan,fill opacity=0.35,draw=none] (7.5,-1.5)--(-7.5,-1.5)--(-7.5,2)-- (7.5,2);
	\node[draw=none] at (-1.5,1) {{\small $\mathcal{S}$}};
	\node at (0,-11) {Figure $(i)$};
	\end{tikzpicture}	
	\hspace{1cm}
	\begin{tikzpicture}[scale=0.175]\label{fig2}
	\draw[->] (-7.5,0)--(8.5,0) node[right]{$u$};
	\draw[->] (0,-7.5)--(0,8.5) node[above]{$v$};
	\draw[red, densely dashdotted] (-4.5,-7.5)--(-4.5,8.5)node[above]{$u=u_0$};
	\draw[blue](-7.5,-1.5)-- (7.5,-1.5) node[below]{$v=-b_2$};
	\draw[blue]  (-7.5,2)-- (7.5,2) node[above]{$v=a_2$};
	\draw[fill=cyan,fill opacity=0.5,draw=none] (7.5,-1.5)--(-7.5,-1.5)--(-7.5,2)-- (7.5,2);
	\node[draw=none] at (-2,1) {{\small $\mathcal{S}'$}};
	\node at (0,-11) {Figure $(ii)$};
	\end{tikzpicture}
\end{center}

\begin{proof}
Set $R^X:=C[X_1, \ldots, X_n]$ and put $R^Y=C[Y_1, \ldots, Y_m]$. Consider $R=S^X:=R^X[Y_1, \ldots, Y_m]$ as standard graded with $\deg z=0$ for all $z \in R^X$ and $\deg Y_j=1$ for $j=1, \ldots,m$. Then by \cite[Lemma 5.5]{TPSR}, $\mathcal{T}(S^X)=\bigoplus_{v\in \Z} M_{\star, v}$ is a countable generated graded $S^X$-module. Similarly consider $R=S^Y:=R^Y[X_1, \ldots X_n]$ as standard graded with $\deg z=0$ for all $z \in R^Y$, $\deg X_i=1$ for $i=1 \ldots n$. Then we get that $\mathcal{T}(S^Y)=\bigoplus_{ u\in \Z} M_{u, \star}$ is a countable generated graded $S^Y$-module. So by Lemma \ref{lem1}, $M_{\star, v}=\bigoplus_{ u\in \Z} M_{(u,v)}$ (resp. $M_{u, \star}=\bigoplus_{ v\in \Z} M_{(u,v)}$) is a countably generated $R^X$-module (resp. $R^Y$-module). Moreover, from Remark \ref{comp-grd-genEur} we get $M_{\star, v}$ (resp. $M_{u, \star}$) are graded generalized Eulerian $A_n(C)$-modules (resp. $A_m(C)$-modules). Fix $v \in \mathbb{Z}$. From \cite[Corollary 5.12]{TPSR} it follows that ``if $M_{\star, v} \neq 0$ then $(M_{\star, v})_u=M_{(u,v)} \neq 0$ for infinitely many $u \gg 0$ OR $M_{(u,v)} \neq 0$ for infinitely many $u \ll 0$". Therefore, for some $a,b \in \Z$ with $a<b$, if $M_{(u,v)}=0$ for all $u \notin [a,b]$, then $M_{\star, v}=0$. The same statement holds if we interchange $u$ and $v$.
	
We now consider the shaded region $\mathcal{S}$ in Figure (i) and let $M_{(u,v)}=0$ if $(u,v) \notin \mathcal{S}$. For each fixed $v_0> a_2$ or $v_0< -b_2$, as $(M_{*,v_0})_u=M_{(u,v_0)}=0$ for all $u \gg0$ and $u \ll 0$ so it follows that $M_{*,v_0}=0$. Now we have the following situation: $M_{(u,v)}=0$ if $(u,v) \notin \mathcal{S}'$, the new shaded part in Figure $(ii)$. Fix $u_0 \in \Z$. As $(M_{u_0, *})_v=M_{(u_0,v)}=0$ for all $v \gg0$ and $v \ll 0$ so we get that $M_{u_0, *}=0$. Hence $M_{(u,v)}=0$ for all $(u,v) \in \Z^2$. The result follows.  
\end{proof}

\s{\bf Tameness property.}\\ 
We now prove the following result which shows tameness property of $M$.

\begin{theorem}[Tameness]{\rm(}With hypotheses as in \ref{setup_bigraded}{\rm)}. If $M \neq 0$, then there exists at least one region $\mathcal{V}$ as in Figure (iv) such that $M_{(u,v)} \neq 0$ for all $(u,v) \in \mathcal{V}$.
\end{theorem}

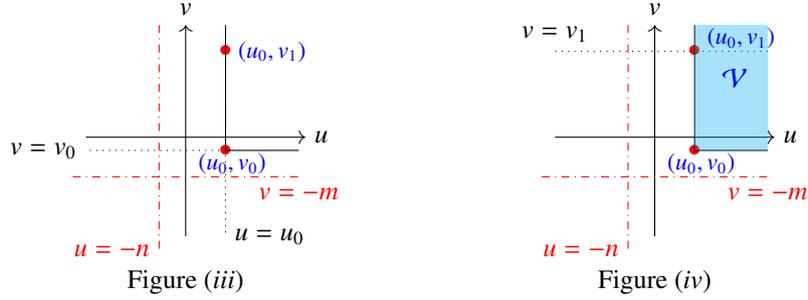
	\begin{center}
	\begin{tikzpicture}[scale=0.175]
	\draw[->] (-7.5,0)--(9,0) node[right]{$u$};
	\draw[->] (0,-7.5)--(0,8.5) node[above]{$v$};
	\draw[dashdotted, red] (-2,8.5)--(-2,-8.5) node[left]{$u=-n$};
	\draw[dashdotted, red] (-8.5,-3)--(8.5, -3) node[below]{$v=-m$};
	\node[red] at (3,-1) {\textbullet};
	\node[red] at (3,6.5) {\textbullet};
	\draw[-](3,8.5)--(3,-1)-- (8.5,-1);
	\draw[dotted] (3,-1)--(-7.5,-1)node[left]{$v=v_0$};
	\draw[dotted] (3,2.5)--(3,-7.5)node[right]{$u=u_0$};
	\node[blue,draw=none] at (3.5,-2) {{\small $(u_0,v_0)$}};
	\node[blue,draw=none] at (6.5,6.5) {{\small $(u_0,v_1)$}};
	\node at (0,-11) {Figure $(iii)$};
	\end{tikzpicture}
	\hspace{2cm}
	\begin{tikzpicture}[scale=0.175]
	\draw[->] (-7.5,0)--(9,0) node[right]{$u$};
	\draw[->] (0,-7.5)--(0,8.5) node[above]{$v$};
	\draw[dashdotted, red] (-2,8.5)--(-2,-8.5) node[left]{$u=-n$};
	\draw[dashdotted, red] (-8.5,-3)--(8.5, -3) node[below]{$v=-m$};
	\node[red] at (3,-1) {\textbullet};
	\node[red] at (3,6.5) {\textbullet};
	\draw[](3,8.5)--(3,-1)-- (8.5,-1);
	\draw[dotted] (8.5,6.5)--(-7.5,6.5)node[above]{$v=v_1$};
	\draw[fill=cyan,fill opacity=0.35,draw=none] (3,8.5)--(3,-1)--(8.5,-1)--(8.5,8.5)--(3,8.5);
	\node[blue,draw=none] at (6,4.5) {$\mathcal{V}$};
	\node[blue,draw=none] at (3.5,-2) {{\small $(u_0,v_0)$}};
	\node[blue,draw=none] at (6.5,7.5) {{\small $(u_0,v_1)$}};
	\node at (0,-11) {Figure $(iv)$};
	\end{tikzpicture}
\end{center}

\begin{proof}
	If $M \neq 0$ then $M_{(u_0,v_0)} \neq 0$ for some $(u_0,v_0) \in \Z^2$. Thus $M_{\star,v_0}=\bigoplus_{u \in \Z} M_{(u,v_0)} \neq 0$ and $M_{u_0, \star}=\bigoplus_{v \in \Z} M_{(u_0,v)} \neq 0$. By \cite[Corollary 5.13]{TPSR},	
	\begin{align*}
	& M_{(u_0,v)}\left\{
	\begin{array}{ll}
	\neq 0~ \forall v \geq v_0  & \mbox{if } v_0 \geq -m+1 \\
	\neq 0~ \forall v \leq v_0  & \mbox{if } v_0 \leq -m 
	\end{array}
	\right. \mbox{ and }
	& M_{(u,v_0)}\left\{
	\begin{array}{ll}
	\neq 0~ \forall u \geq u_0  & \mbox{if } u_0 \geq -n+1 \\
	\neq 0~ \forall u \leq u_0  & \mbox{if } u_0 \leq -n 
	\end{array}
	\right.
	\end{align*}
In particular, if $u_0 \geq -n+1$ and $v_0 \geq -m+1$ then we get Figure $(iii)$. For instance, pick $v_1 \geq v_0$. Since $M_{(u_0,v_1)} \neq 0$, we have $M_{\star, v_1}=\bigoplus_{u \in \Z} M_{(u,v_1)} \neq 0$. Now as $u_0 \geq -n+1$ so by \cite[Corollary 5.13]{TPSR} $(M_{\star, v_1})_{u}=M_{u, v_1} \neq 0$ for all $u \geq u_0$. Varying $v_1$ from $v_0$ to $\infty$, we get the shaded region $\mathcal{V}$ in Figure $(iv)$ such that $M_{(u,v)} \neq 0$ for all $(u,v) \in \mathcal{V}$. In the same way, using \cite[Corollary 5.13]{TPSR} we will get different regions depending on the values of $(u_0,v_0)$.
\end{proof}

\s{ \bf Rigidity property.}\\
In this subsection, we explore the rigidity property of components of a bigraded generalized Eulerian $A_{n,m}(C)$-module $M$. This property is stronger than those in Sub-section \ref{van_tem} and plays a crucial role in analyzing components of bigraded local cohomology modules. Using this property, we get some interesting information regarding bass numbers, associated primes, dimensions of supports, and injective dimensions of biigraded components in Section 6.

\begin{theorem}[Rigidity-1]\label{bi-rigid} {\rm(}With hypotheses as in \ref{setup_bigraded} and notations as in Subsection \ref{notations}{\rm )}. Let $\mathcal{R}$ be one of the regions in the set $\big\{\mathcal{NE}, \mathcal{NW^*}, \mathcal{S^*W^*}, \mathcal{S^*E}\big\}$. Then the following conditions are equivalent: 
	\begin{enumerate}[\rm(a)]
		\item There exists $(u_0,v_0)\in \mathcal{R}$ such that $M_{(u,v)} \neq 0$ for all $(u,v) \in \mathfrak{B}_{(u_0,v_0)}$.
		\item $M_{(u,v)} \neq 0$ for all $(u,v)\in \mathcal{R}$.
		\item $M_{(u_1,v_1)} \neq 0$ for some $(u_1,v_1) \in \mathcal{R}$.
	\end{enumerate}
\end{theorem}

\begin{note*}
One can give a short proof of the above result using the same arguments as in Theorem \ref{bi-rigid2}. But the following relatively longer proof will help the reader to get an insight into the statements.	
\end{note*}

\begin{proof}
We fix $\mathcal{R}$ to be the region $\mathcal{NE}$. Clearly $(b) \implies (a) \implies (c)$. We  only have to prove $(c) \implies (b)$. We use induction on $r=n+m$. As $n, m \geq 1$, so $r \geq 2$. Let $r=2$ which holds only when $n=1=m$. 
	
	Now $H_1(Y_1, M)_{(u,v)}$ is nonzero only if $v=0$ and $H_1(X_1, M)_{(u,v)}$ is nonzero only if $u=0$ for $l=0,1$. We have the following exact sequences
	\begin{equation}\label{rigid0}
	\begin{split}
	(1) \quad & 0 \to H_1(Y_1, M)_{(u,v)} \to M_{(u,v-1)} \overset{\cdot Y_1}{\lrt} M_{(u,v)} \to H_0(Y_1, M)_{(u,v)} \to 0,\\
	(2) \quad & 0 \to H_1(X_1, M)_{(u,v)} \to M_{(u-1,v)} \overset{\cdot X_1}{\lrt} M_{(u,v)} \to H_0(X_1, M)_{(u,v)} \to 0.
	\end{split}
	\end{equation}
	Thus we get that $M_{(u,v)} \cong M_{(u,v-1)}$ for all $v \geq 1$, that is $M_{(u,v)} \cong M_{(u,0)} \cong M_{(u,v_1)}$ for all $(u,v) \in \Z$ with $v \geq 0$. Similarly from the second equation we get that $M_{(u,v)} \cong M_{(0,v)} \cong M_{(u_1,v)}$ for all $(u,v) \in \Z$ with $u \geq 0$. Hence for each $(u,v) \in \mathcal{NE}$,  $M_{(u,v)} \cong M_{(u_1,v_1)} \neq 0$.
	
	Now let $r>2$. For $j=0,1$, we have $H_j(Y_m, M)$ and $H_j(X_n, M)$ are bigraded generalized Eulerian $A_{n,m-1}(C)$-module and $A_{n-1,m}(C)$-module, respectively. 
	
	We separately deal with the case when either $n=1$ or $m=1$. Let $n=1$, then $H_l(X_n, M)_{(u,v)}$ is nonzero only if $u=0$. From the exact sequence 
	\begin{equation*}
	0 \to H_1(X_1, M)_{(u,v)} \to M_{(u-1,v)} \overset{\cdot X_1}{\lrt} M_{(u,v)} \to H_0(X_1, M)_{(u,v)} \to 0,
	\end{equation*}
	we get that $M_{(u,v)} \cong M_{(u-1,v)} \cong M_{(0,v)} \cong M_{(u_1,v)}$ for all $(u,v) \in \Z$ with $u \geq 0$. Recall that $M_{\star,v_1}:=\bigoplus_{u \in \Z}M_{(u,v_1)}$ is a graded generalized Eulerian $A_1(C)$-module with $(M_{\star,v_1})_{u_1}= M_{(u_1, v_1)}\neq 0$ for some $u_1 \geq 0$. Hence by \cite[Theorem 7.1 (II)]{TPSR} it follows that $(M_{\star,v_1})_{u}=M_{(u,v_1)} \neq 0$ for all $u \geq 0$. Choose $u \in (-\infty, -1]$ and fix it. Since $M_{u,\star}:=\bigoplus_{v \in \Z}M_{(u,v)}$ is a graded generalized Eulerian $A_m(C)$-module with $(M_{u,\star})_{v_1}= M_{(u, v_1)}\neq 0$ for some $v_1 \geq 0$, so again by \cite[Theorem 7.1 (II)]{TPSR} we get that $(M_{u,\star})_{v}=M_{(u,v)} \neq 0$ for all $v \geq 0$. 
	
	Now let $n>1$ and $m>1$. This allows us to use induction hypothesis on $H_1(Y_m, M)$ and $H_1(X_n, M)$ (as $n-1 \geq 1$ and $m-1 \geq 1$). We have the following exact sequences 
	\begin{equation}\label{rigid1}
	\begin{split}
	(1) \quad &0 \to H_1(Y_m, M)_{(u,v)} \to M_{(u,v-1)} \overset{\cdot Y_m}{\lrt} M_{(u,v)} \to H_0(Y_m, M)_{(u,v)} \to 0,\\
	(2) \quad & 0 \to H_1(X_n, M)_{(u,v)} \to M_{(u-1,v)} \overset{\cdot X_n}{\lrt} M_{(u,v)} \to H_0(X_n, M)_{(u,v)} \to 0.
	\end{split}
	\end{equation}
	Now consider the following three cases:
	
	\vspace{0.25cm}
	{\it Case 1:} $H_0(Y_m; M)_{(u_2,v_2)} \neq 0$ for some $(u_2,v_2) \in \mathcal{NE}$.
	
	By the induction hypothesis it follows that $H_0(Y_m; M)_{(u,v)} \neq 0$ for all $(u,v) \geq 0$. So by the exact sequence \eqref{rigid1}(1) we get that $M_{(u,v)} \neq 0$ for all $(u,v)\geq (0,0)$.
	
	\vspace{0.25cm}
	{\it Case 2:} $H_1(Y_m; M)_{(u_2,v_2)} \neq 0$ for some $(u_2,v_2) \in \mathcal{NE}$.
	
	By the induction hypothesis we get that $H_1(Y_m; M)_{(u,v)} \neq 0$ for all $(u,v) \geq 0$. So by the exact sequence \eqref{rigid1}(1) we get that $M_{(u,v)} \neq 0$ for all $(u,v)\geq (0,-1)$ and hence for all $(u,v)\geq (0,0)$.
	
	\vspace{0.25cm}
	{\it Case 3:} $H_0(X_n; M)_{(u_2,v_2)} \neq 0$ for some $(u_2,v_2) \in \mathcal{NE}$.
	
	By the induction hypothesis it follows that $H_0(Y_m; M)_{(u,v)} \neq 0$ for all $(u,v) \geq 0$. So by the exact sequence \eqref{rigid1}(2) we get that $M_{(u,v)} \neq 0$ for all $(u,v)\geq (0,0)$.
	
	\vspace{0.25cm}
	{\it Case 4:} $H_1(X_n; M)_{(u_2,v_2)} \neq 0$ for some $(u_2,v_2) \in \mathcal{NE}$.
	
	By the induction hypothesis we get that $H_1(Y_m; M)_{(u,v)} \neq 0$ for all $(u,v) \geq 0$. So by the exact sequence \eqref{rigid1}(2) we get that $M_{(u,v)} \neq 0$ for all $(u,v)\geq (-1,0)$ and hence for all $(u,v)\geq (0,0)$.
	
	\vspace{0.25cm}
	{\it Case 5:} For $j=0, 1$; $H_j(Y_m; M)_{(u_2,v_2)}= 0$ and $H_j(X_n; M)_{(u_2,v_2)}= 0$ for all $(u_2,v_2) \in \mathcal{NE}$.
	
	From \eqref{rigid1} we get that $M_{(u,v-1)} \cong M_{(u,v)} \cong M_{(u-1,v)}$ for all $(u,v)\geq (0,0)$. Therefore $M_{(u,v)} \cong M_{(u_1,v_1)} \neq 0$ for all $(u,v)\geq (0,0)$. The result follows.

	\vspace{0.15cm}
	For the region $\mathcal{NW^*}$: use the following pair of short exact sequences 
	\begin{equation}\label{rigid2}
	\begin{split}
	(1) \quad &0 \to H_1(Y_m, M)_{(u,v)} \to M_{(u,v-1)} \overset{\cdot Y_m}{\lrt} M_{(u,v)} \to H_0(Y_m, M)_{(u,v)} \to 0,\\
	(2) \quad & 0 \to H_1(\partial_n, M)_{(u,v)} \to M_{(u+1,v)} \overset{\cdot \partial_n}{\lrt} M_{(u,v)} \to H_0(\partial_n, M)_{(u,v)} \to 0.
	\end{split}
	\end{equation}
	
	\vspace{0.5cm}
	For the region $\mathcal{S^*W^*}$: use the following pair of short exact sequences 
	\begin{equation}\label{rigid3}
	\begin{split}
	(1) \quad &0 \to H_1(\delta_m, M)_{(u,v)} \to M_{(u,v+1)} \overset{\cdot \delta_m}{\lrt} M_{(u,v)} \to H_0(\delta_m, M)_{(u,v)} \to 0,\\
	(2) \quad & 0 \to H_1(\partial_n, M)_{(u,v)} \to M_{(u+1,v)} \overset{\cdot \partial_n}{\lrt} M_{(u,v)} \to H_0(\partial_n, M)_{(u,v)} \to 0.
	\end{split}
	\end{equation}
	
	\vspace{0.5cm}
	For the region $\mathcal{S^*E}$: use the following pair of short exact sequences 
	\begin{equation}\label{rigid4}
	\begin{split}
	(1) \quad &0 \to H_1(\delta_m, M)_{(u,v)} \to M_{(u,v+1)} \overset{\cdot \delta_m}{\lrt} M_{(u,v)} \to H_0(\delta_m, M)_{(u,v)} \to 0,\\
	(2) \quad & 0 \to H_1(X_n, M)_{(u,v)} \to M_{(u-1,v)} \overset{\cdot X_n}{\lrt} M_{(u,v)} \to H_0(X_n, M)_{(u,v)} \to 0.
	\end{split}
	\end{equation}
	\end{proof}

We now discuss the second part of the rigidity properties. For this, we will continue with the numbering in Theorem \ref{bi-rigid}.

\begin{theorem}[Rigidity-2]\label{bi-rigid2}{With hypothesis as in \ref{setup_bigraded} and notations as in Subsection \ref{notations}{\rm)}}. Pick a coordinate $(u_0,v_0)$ from one of the truncated regions in the set 
\[\big\{\mathrm{Trun}(\mathcal{N}), \mathrm{Trun}(\mathcal{W^*}), \mathrm{Trun}(\mathcal{S^*}), \mathrm{Trun}(\mathcal{E})\big\}.\]
Suppose that $M_{(u_0,v_0)}\neq 0$. Then $M_{(u,v)} \neq 0$ for all $(u,v)$ in the respective region in the set $\big\{\mathcal{N}, \mathcal{W^*}, \mathcal{S^*}, \mathcal{E}\big\}$, keeping the ordering same. Moreover, if $M_{(u_0,v_0)} \neq 0$ for some $-n<u_0<0$ and $-m< v_0<0$, then $M_{(u,v)}\neq 0$ for all $(u,v)\in \Z^2$.
\end{theorem}

\begin{proof} 
For the region $\mathcal{N}$, we only need to prove $(b) \implies (a)$. Recall that $M_{\star,v_0}=\bigoplus_{u \in \Z}M_{(u,v_0)}$ is a graded generalized Eulerian $A_n(C)$-module and $(M_{\star,v_0})_{u_0}=M_{(u_0,v_0)} \neq 0$ for some $-n<u_0< 0$. So by \cite[Theorem 7.3]{TPSR} it follows that $M_{(u,v_0)} \neq 0$ for all $u \in \Z$, see Figure $\mathcal{N} (b)$ (red line). Now for any fixed $u \in \Z$, $M_{u, \star}=\bigoplus_{v \in \Z}M_{(u,v)}$ is a graded generalized Eulerian $A_m(C)$-module and $(M_{u, \star})_{v_0}=M_{(u,v_0)} \neq 0$ for some $v_0 \geq 0$. By \cite[Theorem 7.1 (II)]{TPSR} we get that $(M_{u, \star})_v=M_{(u,v)} \neq 0$ for all $v \geq 0$, see Figure $\mathcal{N} (b)$ (blue line). 
	\begin{center}		
				\begin{tikzpicture}[scale=0.12]
		\draw[->] (-7.5,0)--(8.5,0) node[right]{$u$};
		\draw[->] (0,-7.5)--(0,8.5) node[above]{$v$};
		\draw[blue] (-2,3)--(-7.5,3);
		\draw[dotted] (-7.5,-3)--(8,-3) node[below]{$v=-m$};
		\draw[red] (-5,-7.5)--(-5,8);
		\draw[dotted] (-2,-7.5)--(-2,8) node[left]{$u=-n$};
		\node at (-5,-1.5){$\bullet$};
		\node at (0,-11) {\footnotesize{\textit Figure: $\mathcal{W^*}(b)$}};
		\end{tikzpicture}
		\begin{tikzpicture}[scale=0.12]
		\draw[->] (-7.5,0)--(8.5,0) node[right]{$u$};
		\draw[->] (0,-7.5)--(0,8.5) node[above]{$v$};
		\draw[dotted] (-7.5,-3)--(8,-3)node[below]{$v=-m$};
		\draw[dotted] (-2,-7.5)--(-2,8) node[left]{$u=-n$};
		\node at (-5,-1.5){$\bullet$};
		\draw[fill=cyan,fill opacity=0.35,draw=none] (-8,8)--(-2,8)--(-2,-8)--(-8,-8);
		\node at (0,-11) {\footnotesize{\textit Figure: $\mathcal{W^*}(a)$}};
		\end{tikzpicture}
		\begin{tikzpicture}[scale=0.12]
		\draw[->] (-7.5,0)--(8.5,0) node[right]{$u$};
		\draw[->] (0,-7.5)--(0,8.5) node[above]{$v$};
		\draw[red] (8,3)--(-7.5,3);
		\draw[dotted] (-7.5,-3)--(8,-3)node[below]{$v=-m$};
		\draw[blue] (5,0)--(5,8);
		\draw[dotted] (-2,-7.5)--(-2,8) node[left]{$u=-n$};
		\node at (-1,3){$\bullet$};
		\node at (0,-11) {\footnotesize{\textit Figure: $\mathcal{N} (b)$}};
		\end{tikzpicture}
		\begin{tikzpicture}[scale=0.12]
		\draw[->] (-7.5,0)--(8.5,0) node[right]{$u$};
		\draw[->] (0,-7.5)--(0,8.5) node[above]{$v$};
		\draw[dotted] (-7.5,-3)--(8,-3)node[below]{$v=-m$};
		\draw[dotted] (-2,-7.5)--(-2,8) node[left]{$u=-n$};
		\node at (-1,3){$\bullet$};
		\draw[fill=cyan,fill opacity=0.35,draw=none] (-8,8)--(8,8)--(8,0)--(-8,0);
		\node at (0,-11) {\footnotesize{\textit Figure: $\mathcal{N} (a)$}};
		\end{tikzpicture}				
	\end{center}

For the region $\mathcal{W^*}$, we only need to prove $(b) \implies (a)$. Recall that $M_{u_0, \star}=\bigoplus_{v \in \Z}M_{(u_0,v)}$ is a graded generalized Eulerian $A_m(C)$-module and $(M_{u_0, \star})_{v_0}=M_{(u_0,v_0)} \neq 0$ for some $-m<v_0< 0$. So by \cite[Theorem 7.3]{TPSR} it follows that $M_{(u_0,v)} \neq 0$ for all $v \in \Z$, see Figure $\mathcal{W^*}(b)$ (red line). Now for any fixed $v \in \Z$, $M_{\star, v}=\bigoplus_{u \in \Z}M_{(u,v)}$ is a graded generalized Eulerian $A_n(C)$-module and $(M_{\star, v})_{u_0}=M_{(u_0,v)} \neq 0$ for some $u_0 \leq -n$. By \cite[Theorem 7.1 (I)]{TPSR} we get that $(M_{\star, v})_u=M_{(u,v)} \neq 0$ for all $u \leq -n$, see Figure $\mathcal{W^*}(b)$ (blue line). 
	\begin{center}	
		\begin{tikzpicture}[scale=0.12]
		\draw[->] (-7.5,0)--(8.5,0) node[right]{$u$};
		\draw[->] (0,-7.5)--(0,8.5) node[above]{$v$};
		\draw[red] (7.5,-5)--(-7.5,-5);
		\draw[dotted] (8,-3)--(-7.5,-3)node[above]{$v=-m$};
		\draw[blue] (5,-7.5)--(5,-3);
		\draw[dotted] (-2,-7.5)--(-2,8) node[left]{$u=-n$};
		\node at (-1,-5){$\bullet$};
		\node at (0,-11) {\footnotesize{\textit Figure: $\mathcal{S^*}(b)$}};
		\end{tikzpicture}
		\begin{tikzpicture}[scale=0.12]
		\draw[->] (-7.5,0)--(8.5,0) node[right]{$u$};
		\draw[->] (0,-7.5)--(0,8.5) node[above]{$v$};
		\draw[dotted] (-7.5,-3)--(8,-3)node[below]{$v=-m$};
		\draw[dotted] (-2,-7.5)--(-2,8) node[left]{$u=-n$};
		\node at (-1,-5){$\bullet$};
		\draw[fill=cyan,fill opacity=0.35,draw=none] (-8,-3)--(8,-3)--(8,-8)--(-8,-8);
		\node at (0,-11) {\footnotesize{\textit Figure: $\mathcal{S^*}(a)$}};
		\end{tikzpicture}			
		\begin{tikzpicture}[scale=0.12]
		\draw[->] (-7.5,0)--(8.5,0) node[right]{$u$};
		\draw[->] (0,-7.5)--(0,8.5) node[above]{$v$};
		\draw[blue] (0,3)--(8,3);
		\draw[dotted] (8,-3)--(-7.5,-3)node[left]{$v=-m$};
		\draw[red] (5,-7.5)--(5,8);
		\draw[dotted] (-2,-7.5)--(-2,8) node[left]{$u=-n$};
		\node at (5,-1.5){$\bullet$};
		\node at (0,-11) {\footnotesize{\textit Figure: $\mathcal{E}(b)$}};
		\end{tikzpicture}
		\begin{tikzpicture}[scale=0.12]
		\draw[->] (-7.5,0)--(8.5,0) node[right]{$u$};
		\draw[->] (0,-7.5)--(0,8.5) node[above]{$v$};
		\draw[dotted] (8,-3)--(-7.5,-3)node[left]{$v=-m$};
		\draw[dotted] (-2,-7.5)--(-2,8) node[left]{$u=-n$};
		\node at (5,-1.5){$\bullet$};
		\draw[fill=cyan,fill opacity=0.35,draw=none] (0,8)--(8,8)--(8,-7.5)--(0,-7.5);
		\node at (0,-11) {\footnotesize{\textit Figure: $\mathcal{E}(a)$}};
		\end{tikzpicture}
	\end{center}
	Similarly as above, one can prove the assertions for the remaining regions.

\vspace{0.2cm}

\noindent
\begin{minipage}{0.445\textheight}
	
Finally, for the last assertion, we only need to prove $(ii) \implies (i)$. Recall that $M_{\star,v_0}=\bigoplus_{u \in \Z}M_{(u,v_0)}$ is a graded generalized Eulerian $A_n(C)$-module and $(M_{\star,v_0})_{u_0}=M_{(u_0,v_0)} \neq 0$ for some $-n<u_0< 0$. So by \cite[Theorem 7.3]{TPSR} it follows that $M_{(u,v_0)} \neq 0$ for all $u \in \Z$, see Figure $uv\mbox{-plane} (b)$ (red line). Now for any fixed $u \in \Z$, $M_{u, \star}=\bigoplus_{v \in \Z}M_{(u,v)}$ is a graded generalized Eulerian $A_m(C)$-module and $(M_{u, \star})_{v_0}=M_{(u,v_0)} \neq 0$ for some $-m<v_0< 0$. So again by \cite[Theorem 7.3]{TPSR} we get that $M_{(u,v)} \neq 0$ for all $v \in \Z$, see Fig. $uv\mbox{-plane} (b)$.
\end{minipage}
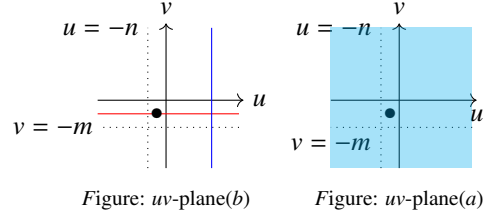
\begin{minipage}{0.3\textheight}
\begin{center}
		\begin{tikzpicture}[scale=0.12]
\draw[->] (-7.5,0)--(8.5,0) node[right]{$u$};
\draw[->] (0,-7.5)--(0,8.5) node[above]{$v$};
\draw[red] (8,-1.5)--(-7.5,-1.5);
\draw[dotted] (8,-3)--(-7,-3)node[left]{$v=-m$};
\draw[blue] (5,-7.5)--(5,8);
\draw[dotted] (-2,-7.5)--(-2,8) node[left]{$u=-n$};
\node at (-1,-1.5){$\bullet$};
\node at (0,-11) {\footnotesize{\textit Figure: $uv\mbox{-plane}(b)$}};
\end{tikzpicture}
\begin{tikzpicture}[scale=0.12]
\draw[->] (-7.5,0)--(8.5,0) node[below]{$u$};
\draw[->] (0,-7.5)--(0,8.5) node[above]{$v$};
\draw[dotted] (8,-3)--(-7.5,-3) node[below]{$v=-m$};
\draw[dotted] (-2,-7.5)--(-2,8) node[left]{$u=-n$};
\node at (-1,-1.5){$\bullet$};
\draw[fill=cyan,fill opacity=0.35,draw=none] (-7.5,8)--(8,8)--(8,-7.5)--(-7.5,-7.5);
\node at (0,-11) {\footnotesize{\textit Figure: $uv\mbox{-plane} (a)$}};
\end{tikzpicture}
\end{center}
\end{minipage}

\end{proof}

\s We now discuss a few examples whose nonzero components are associated the shaded regions in the $uv$-plane in the above figures. 

\vspace{0.15cm}
Let $R = K[X_1,\ldots, X_n]$ be a polynomial ring over a field $K$. Then
\[ H^{n}_{(X_1, \ldots, X_n)}(R)=E_R(K) \cong X_1^{-1}\cdots X_n^{-1}K[X_1^{-1}, \ldots, X_n^{-1}],\]
see \cite[p. 277]{BS}. Take the ideal $\mathfrak{a}=(X_1, \ldots, X_t)$ in $S:=K[X_1, \ldots, X_t]$ for some $t \leq n$. Then 
\begin{align*}
H^t_{\mathfrak{a}}(R)&=H^t_{\mathfrak{a}}(S) \otimes_{S} R\\
&\cong X_1^{-1}\cdots X_t^{-1}K[X_1^{-1}, \ldots, X_t^{-1}] \otimes_K K[X_{t+1}, \ldots, X_n]\\
&\cong X_1^{-1}\cdots X_t^{-1}K[X_1^{-1}, \ldots, X_t^{-1},X_{t+1}, \ldots, X_n].
\end{align*}

\begin{example}\label{eg_rigidity}
Let $R = K[X_1,\ldots, X_n, Y_1, \ldots, Y_m]$ be a polynomial ring over a field $K$ with $\bideg X_i=(1,0)$ and $\bideg Y_j=(0,1)$ for $i=1, \ldots, n$ and $j=1, \ldots, m$. Take the ideals $\m=(X_1,\ldots, X_n, Y_1, \ldots, Y_m)$, $\mathfrak{a}=(X_1,\ldots, X_n)$, $\mathfrak{b}=(Y_1, \ldots, Y_m)$, $\mathfrak{c}=(X_1,\ldots, X_n,Y_1)$, and $\mathfrak{d}=(X_n,Y_1, \ldots, Y_m)$ in $R$. Then one can verify the following:
	\begin{enumerate}[\rm(i)]
	\item $H^1_{(X_1)}(R)\cong X_1^{-1}K[X_1^{-1},X_2, \ldots, X_n,Y_1, \ldots, Y_m]$,
	\item $H^1_{(Y_1)}(R)\cong Y_1^{-1}K[X_1, \ldots, X_n,Y_1^{-1},Y_2, \ldots, Y_m]$,
	\item $H^2_{(X_1,Y_1)}(R)\cong (X_1^{-1}Y_1^{-1})K[X_1^{-1},X_2, \ldots, X_n,Y_1^{-1},Y_2, \ldots, Y_m]$,
	\item $H^n_{\mathfrak{a}}(R)\cong (X_1^{-1}\cdots X_n^{-1})K[X_1^{-1}, \ldots, X_n^{-1},Y_1, \ldots, Y_m]$,
	\item $H^m_{\mathfrak{b}}(R)\cong (Y_1^{-1}\cdots Y_m^{-1})K[X_1,\ldots, X_n,Y_1^{-1}\cdots Y_m^{-1}]$,
	\item $H^{n+1}_{\mathfrak{c}}(R)\cong (X_1^{-1}\cdots X_n^{-1}Y_1^{-1})K[X_1^{-1}, \ldots, X_n^{-1},Y_1^{-1},Y_2, \ldots, Y_m]$,
	\item $H^{m+1}_{\mathfrak{b}}(R)\cong (X_n^{-1}Y_1^{-1}\cdots Y_m^{-1})K[X_1,\ldots, X_{n-1},X_n^{-1},Y_1^{-1}\cdots Y_m^{-1}]$,
	\item $H^{n+m}_{\m}(R)\cong (X_1^{-1}\cdots X_n^{-1}Y_1^{-1}\cdots Y_m^{-1})K[X_1^{-1}, \ldots, X_n^{-1}, Y_1^{-1}, \ldots, Y_m^{-1}]$.
	\end{enumerate}
	For any arbitrary integers $i_1, \ldots, i_n$ and $j_1, \ldots, j_m$,
	\begin{align*}
	&\e^n_X \cdot X_1^{i_1} \cdots X_n^{i_n}= (i_1+ \cdots+i_n)X_1^{i_1} \cdots X_n^{i_n} \\
	\text{and} \quad&  \e^m_Y \cdot Y_1^{j_1} \cdots Y_m^{j_m}= (j_1+ \cdots+j_m) Y_1^{j_1} \cdots Y_m^{j_m}.
	\end{align*}
	So $H^{0}_{(0)}(R)=R, H^{n}_{\mathfrak{a}}(R), H^{n+m}_{\m}(R)$, $H^{m}_{\mathfrak{b}}(R)$, $H^2_{(X_1,Y_1)}(R)$, $H^1_{(X_1)}(R)$, $H^{n+1}_{\mathfrak{c}}(R)$, $H^{m+1}_{\mathfrak{b}}(R)$, and $H^1_{(Y_1)}(R)$ all are bigraded Eulerian $A_{n,m}(K)$-modules. Note that they are represented by the regions $\mathcal{NE}, \mathcal{NW^*}, \mathcal{S^*W^*}, \mathcal{S^*E}, uv\mbox{-plane}, \mathcal{N}, \mathcal{W^*},\mathcal{S^*}, \mathcal{E}$ respectively, especially when $n,m \geq 2$.
	
	\vspace{0.15cm}
	As $H^j_{\mathfrak{a}}(S)=0$ for all $j \neq n$ so we have $H^j_{\mathfrak{a}}(R)=0$ for all $j \neq n$. Similarly $H^j_{\mathfrak{b}}(R)=0$ for all $j \neq m$, and $H^j_{\mathfrak{m}}(R)=0$ for all $j \neq n+m$. Let $m \geq n \geq 2$. Note that $\mathfrak{a}+\mathfrak{b}=\mathfrak{m}$. Consider the following {\it Mayer-Victoris sequence} 
	\[\cdots \to H^n_{\mathfrak{m}}(R) \to H^n_{\mathfrak{a}}(R) \oplus H^n_{\mathfrak{b}}(R) \to H^n_{\mathfrak{a}\cap \mathfrak{b}}(R) \to H^{n+1}_{\mathfrak{m}}(R) \to \cdots \]
	Since $n, n+1< n+m$, we get that if $m=n$ then $H^n_{\mathfrak{a}}(R) \oplus H^n_{\mathfrak{b}}(R) \cong H^n_{\mathfrak{a}\cap \mathfrak{b}}(R)$ is represented by the combination of Figures (II) and (IV). Whereas, if $m >n$ then $H^n_{\mathfrak{a}}(R)\cong H^n_{\mathfrak{a}\cap \mathfrak{b}}(R)$ is represented by Figure (II).
\end{example}
\begin{note*}
	Take the ideal $I=(X+Y)$ in $R=K[X,Y]$. As $R=K[X,X+Y] \cong K[Z_1,Z_2]:=R'$ where $Z_1,Z_2$ are variables, so 
	\[H^1_I(R) \cong H^1_{(Z_2)}(R') \cong Z_2^{-1}K[Z_1,Z_2^{-1}] \cong (X+Y)^{-1}K[X,(X+Y)^{-1}].\] 
	Therefore Example \ref{eg_rigidity} can be rewritten for ideals like $I$ which are not monomial ideals. 
\end{note*}

\begin{remark}\label{localize_MGE}
	Let $M$ be a multigraded generalized Eulerian $A_{(n,m)}(C)$-module. Let $f$ be a bihomogeneous element in $R$ with $\bideg f=(u,v)$, and $z$ a bihomogeneous element in $M$ with $\bideg z=(u',v')$. In the same manner as \cite[Lemma 3.5]{TPJS}, one can show that
	\[(\e^X_n-u'+u)\cdot \frac{z}{f}=\frac{1}{f}(\e^X_n-u')\cdot z \quad \mbox{ and } \quad (\e^Y_m-v'+v)\cdot \frac{z}{f}=\frac{1}{f}(\e^Y_m-v')\cdot z.\]
	Then using the same line of proof of \cite[Lemma 3.6, Corollary 3.7]{TPJS}, we obtain that $M_f$ is a bigraded generalized Eulerian $A_{(n,m)}(C)$-module.
\end{remark}
\begin{example}
	Let $R=K[X_1,X_2,Y_1,Y_2]$ be a polynomial ring over a field $K$. Take the ideal $I=(X_1,X_2)$ in $R$. By Example \ref{eg_rigidity}, $H^2_I(R)=X_1^{-1}X_2^{-1} K[X_1^{-1},X_2^{-1},Y_1,Y_2]$. From \cite[Exercise 1.2.7]{BS} we further get that $H^2_I(R)_{Y_1} \cong H^2_{IR_{Y_1}}(R_{Y_1})$. Moreover, $H^2_{IR_{Y_1}}(R_{Y_1}) \cong H^2_{I}(R_{Y_1})$. Since $H^2_I(R)$ is a multigraded generalized Eulerian, holonomic $A_{(2,2)}(K)$-module (see Definition \ref{bi-holo}), so is $H^2_I(R)_{Y_1}$ by Remark \ref{localize_MGE} and \cite[Theorem 3.4.1]{BS}. Observe that
	\begin{align*}
	H^2_I(R)_{Y_1}\cong& H^2_I(R) \otimes_R R_{Y_1}\\
	=&X_1^{-1}X_2^{-1} K[X_1^{-1},X_2^{-1},Y_1,Y_2] \otimes_R \frac{K[X_1,X_2,Y_1,Y_2,W_1]}{(1-Y_1W_1)} \\
	\cong&
	X_1^{-1}X_2^{-1}K[X_1^{-1},X_2^{-1}] \otimes_K K[Y_1,Y_2]\otimes_R\frac{K[X_1,X_2,Y_1,Y_2,W_1]}{(1-Y_1W_1)}\\
	\cong& X_1^{-1}X_2^{-1}K[X_1^{-1},X_2^{-1}] \otimes_K \frac{R}{(X_1,X_2)}\otimes_R\frac{K[X_1,X_2,Y_1,Y_2,W_1]}{(1-Y_1W_1)}\\
	 \cong&X_1^{-1}X_2^{-1}K[X_1^{-1},X_2^{-1}] \otimes_K \frac{K[X_1,X_2,Y_1,Y_2,W_1]}{(X_1,X_2,1-Y_1W_1)}, \quad \mbox{ as } \frac{R}{I} \otimes_R M \cong \frac{M}{IM}\\
	\cong& X_1^{-1}X_2^{-1}K[X_1^{-1},X_2^{-1}] \otimes_K \frac{K[Y_1,Y_2,W_1]}{(1-Y_1W_1)}\\
	\cong& X_1^{-1}X_2^{-1}K[X_1^{-1},X_2^{-1}] \otimes_K K[Y_1,Y_2,W_1] \otimes_{K[Y_1,Y_2,W_1]}\frac{K[Y_1,Y_2,W_1]}{(1-Y_1W_1)}\\
	\cong& X_1^{-1}X_2^{-1} K[X_1^{-1},X_2^{-1},Y_1,Y_2,W_1] \otimes_{K[Y_1,Y_2,W_1]}\frac{K[Y_1,Y_2,W_1]}{(1-Y_1W_1)}\\
	\cong&\frac{X_1^{-1}X_2^{-1} K[X_1^{-1},X_2^{-1}, Y_1,Y_2,W_1]}{(1-Y_1W_1)}\\
	\cong& X_1^{-1}X_2^{-1} K[X_1^{-1},X_2^{-1}, Y_1,Y_2,Y_1^{-1}].
	\end{align*}
	Clearly $H^2_I(R)_{Y_1}$ is represented by the shaded region in Figure (VII) (a) in Theorem \ref{bi-rigid}. Similarly one can check that $H^2_{(Y_1,Y_2)}(R)_{X_1} \cong Y_1^{-1}Y_2^{-1} K[X_1,X_2,X_1^{-1}, Y_1^{-1},Y_2^{-1}]$ is represented by Figure (VIII) (a) in Theorem \ref{bi-rigid}. Further note that $R_{Y_1}=K[X_1,X_2,Y_1,Y_2,Y_1^{-1}]$ and $R_{X_1}=K[X_1,X_2,X_1^{-1},Y_1,Y_2]$ are represented by Figures (IX) (a) (VI) (a) in Theorem \ref{bi-rigid},  respectively.
\end{example}

\section{Polynomial expressions of the \texorpdfstring{$K$}{K}-vector space dimension of the bigraded components}
Suppose that $C$ is a {\it regular ring} containing a field $K$ of characteristic zero. Let $R=C[X_1, \ldots, X_n,$ $ Y_1, \ldots, Y_m]$ be a standard bigraded polynomial ring with $\bideg X_i=(1,0)$ and $\bideg Y_j=(0,1)$ for $i=1, \ldots, n$ and $j=1, \ldots, m$.
 
Let $B=K[[Y_1, \ldots, Y_d]]$ be a power series ring
and $A = B[X_1,\ldots,X_n]$ be a polynomial ring over $B$. Consider $A$ as graded with $\deg b = 0$ for all $b \in B$ and $\deg X_i = 1$ for $i=1,\ldots, n$. Let $D_K(B):= B\langle \partial_1,\ldots, \partial_d\rangle$ be the ring of $K$-linear differential operators on $K$, where $\partial_i=\partial/\partial X_i$ for $i=1,\ldots, n$. By $\D_n := A_n(D_K(B))$, we denote the $n^{th}$-Weyl algebra over $D_K(B)$. Consider $\D_n$ as a graded ring with $\deg a=0$ for all $a \in D_K(B)$, $\deg X_i = 1$ and $\deg \partial_i = -1$ for $i=1,\ldots, n$. Note that $A$ is a graded subring of $\D_n$. 

There is a filtration $\mathcal{B}$ of $\D_n$ such that 
the associated graded ring $\gr_{\mathcal{B}}(\D_n)$ is a polynomial ring over $C$ in $(d + 2n)$-variables, see \cite[4.5]{TP2}. From \cite[3.1.9]{BJ} we get that the weak global dimension of $\D_n$ is $d + n$. Thus a $\D_n$-module $M$ is said to be {\it holonomic} if either it is zero  or there is a $\mathcal{B}$-compatible filtration $\mathcal{F}$ of $M$ such that the associated graded module $\gr_{\mathcal{F}}(M)$ is a finitely generated $\gr_{\mathcal{B}}(\D_n)$-module (equiv. $M$ is finitely generated as a $\D_n$-module, see \cite[Theorem 4.9]{TP2}) of dimension $d + n$. We now extend this notion of holonomic modules to the bigraded setting.

\begin{definition}\label{bi-holo}
If $C=K[[Y_1, \ldots, Y_d]]$, then we say a bigraded $A_{n,m}(D_K(C))$-module $M$ is {\it holonomic} if $\mbox{Tot}(M)$ is holonomic as an $A_{n+m}(D_K(C))$-module.
\end{definition}

\begin{note*}\label{fg}
Clearly, if $\mbox{Tot}(M)$ is zero then $M$ is zero. Further, notice the ring structures of $A_{\underline{s}}(D_K(C))$ and $\mbox{Tot}(A_{n,m}(D_K(C)))$ $=A_{n+m}(D_K(C))$ are the same. Similarly, the module structures of $M$ and $\mbox{Tot}(M)$ are the same. Therefore, if $\mbox{Tot}(M)$ is a finitely generated $A_{n+m}(D_K(C))$-module, then $M$ is a finitely generated $A_{n,m}(D_K(C))$-module. 
\end{note*}

Let $\TT$ be  a bigraded Lyubeznik functor on ${}^*\Mod(R)$. From \cite[Theorem 4.2]{TP2} it immediately follows that $\TT(R)$ is a holonomic $A_{n,m}(D_K(C))$-module.

\vspace{0.15cm}
Throughout this section, $C=K$ is a field. Notice an $A_{n,m}(K)$-module $M$ is holonomic if either it is zero, or it is a finitely generated $A_{n+m}(K)$-module with dimension $n+m$. The general definition of holonomic $A_{n,m}(D_K(C))$-modules is required for the next section.

\vspace{0.15cm}
Across this section, we use the following setup.
\s {\bf Setup}.\label{setup-K-bi} Let $R=K[X_1, \ldots, X_n, Y_1, \ldots, Y_m]$ be a standard bigraded polynomial ring in $(n+m)$-variables over a filed $K$ with $\chars K=0$. Let $A_{n,m}(K)=R\langle \partial_1, \ldots, \partial_n,\delta_1, \ldots, \delta_m\rangle$ be the bigraded $(n,m)^{th}$-Weyl algebra over $K$, where $\partial_i=\partial/\partial X_i$, $\delta_j=\partial/\partial Y_j$ for all $1 \leq i \leq n$ and $1 \leq j \leq m$. Let $M=\oplus_{(u,v)\in \Z^2} M_{(u,v)}$ be a bigraded generalized Eulerian, holonomic $A_{n,m}(K)$-module.

\begin{remark}\label{comp-holo}
Let $M = \bigoplus_{(u,v)\in \Z^2}M_{(u,v)}$ be a holonomic bigraded $A_{n,0}(K)$-module. Fix $v \in \Z$. Set $M_{*,v}:= \bigoplus_{u \in \Z} M_{(u,v)}$. Clearly, $M_{*,v}$ is a bigraded $A_{n,0}(K)$-submodule of $M$. By \cite[Lemma 3.3]{TPSR}, $A_n(K)$ is left (also right) Noetherian. From the sequence $0 \to M_{*,v} \to M$ of left $A_n(K)$-modules we further get that $M_v$ is a finitely generated left $A_n(K)$-module. From \cite[proposition 5.2]{BJ} it then follows that $M_{*,v}$ is a holonomic left $A_n(K)$-module. Similar arguments hold for any holonomic bigraded $A_{0,m}(K)$-module.
\end{remark}

One central result in this section is the following.

\begin{theorem}\label{bi-finite length}
Let $M$ be a holonomic bigraded generalized Eulerian $A_{n,m}(K)$-module with $m,n \geq 1$. 
Suppose that $\dim_K M_{(a_i,b_j)}
<\infty$ for some $a_1,b_1\geq 0$ and some $a_2,b_2\leq -1$. Then $\dim_K M_{(u,v)}<\infty$ for all $(u,v)\in \Z^2$. 
\end{theorem}

We need the following lemma which crucially uses \cite[Theorem 2.12]{TPSR24}.
\begin{lemma}\label{bi-finite length_lemma}
Let $M$ be a holonomic bigraded generalized Eulerian $A_{n,1}(K)$-module with $n \geq 1$. Suppose that $\dim_K M_{(a,b_1)}, \dim_K M_{(a,b_2)}<\infty$ for some $b_1 \geq 0, b_2 \leq -1$ and all $a \in \Z$. Then $\dim_K M_{(u,v)}<\infty$ for all $(u,v)\in \Z^2$. 
\end{lemma}
\begin{proof}
We use induction on $n$. For $n=1$, the result follows from \cite[Theorem 2.12]{TPSR24}. Next, suppose that $n \geq 1$. Consider the exact sequence 
\begin{equation}\label{exact_X}
0 \to H_1(X_n, M)_{(u,v)} \to M_{(u-1, v)} \xrightarrow{\cdot X_n} M_{(u,v)} \to H_0(X_n, M)_{(u,v)} \to 0. 
\end{equation}
From Definition \ref{bi-holo} and \cite[3.4.2, 3.4.4]{BJ} we get that $H_j(X_n, M)$ are bigraded holonomic $A_{n-1,1}(k)$-module for $j=0,1$. Now $\dim_K M_{(a,b)}< \infty$ for $b=b_1, b_2$ implies that 
\[\dim_K H_1(X_n, M)_{(a+1,b)}< \infty \quad \mbox{ and } \quad \dim_K H_0(X_n, M)_{(a,b)}< \infty.\] 
So by induction hypothesis it follows that $\dim_K H_j(X_n, M)_{(u,v)}< \infty$ for all $(u,v) \in \Z^2$ and for $j=0,1$. For $u=a+1$ and $u=a$, from \eqref{exact_X} we get $\dim_K M_{(a+1,b)}< \infty$ and $\dim_K M_{(a-1,b)}< \infty$, respectively. Repeating this process we obtain that $\dim_K M_{(u,b_1)}< \infty$ and $\dim_K M_{(u,b_2)}< \infty$. consider the exact sequence
 \begin{align}
 0 \to H_1(Y_1, M)_{(u,v)} \to M_{(u, v-1)} \xrightarrow{\cdot Y_1} M_{(u,v)} \to H_0(Y_1, M)_{(u,v)} \to 0 \label{Y}
 \end{align}
 Fix $u \in \Z$. Then \eqref{Y} together with Proposition \ref{bigrd-genEur-XY-con} imply that $M_{(u,v)} \cong M_{(u,0)}$ for all $v \geq 0$ and $M_{(u,v)} \cong M_{(u,-1)}$ for all $v \leq -1$. The statement now follows from the assumptions on $b_1, b_2$.
\end{proof}

\begin{proof}[Proof of Theorem \ref{bi-finite length}] We use induction on $r=n+m$. As $m,n \geq 1$ so $r \geq 2$. If $r=2$, then $m=1=n$. Hence $\dim_k M_{(u,v)}< \infty$ by \cite[Theorem 2.12]{TPSR24}.

We now assume that $r > 2$ and the result is true for $<r$. Without loss of generality assume that $n \geq 2$. Then we have an exact sequence 
\begin{equation}\label{fin-dim}
0 \to H_1(X_n, M)_{(u,v)} \to M_{(u-1, v)} \xrightarrow{\cdot X_n} M_{(u,v)} \to H_0(X_n, M)_{(u,v)} \to 0. 
\end{equation}
From Definition \ref{bi-holo} and \cite[3.4.2, 3.4.4]{BJ} we get that $H_j(X_n, M)$ are bigraded holonomic $A_{n-1,m}(k)$-module for $j=0,1$. Now $\dim_K M_{(a,b)}< \infty$ implies that 
\[\dim_K H_1(X_n, M)_{(a+1,b)}< \infty \quad \mbox{ and } \quad \dim_K H_0(X_n, M)_{(a,b)}< \infty.\] 
Then by the induction hypothesis it follows that $\dim_K H_j(X_n, M)_{(u,v)}< \infty$ for all $(u,v) \in \Z^2$ and for $j=0,1$. Putting $u=a+1$ and $u=a$, from \eqref{fin-dim} we have $\dim_K M_{(a+1,b)}< \infty$ and $\dim_K M_{(a-1,b)}< \infty$, respectively. Repeating this process we obtain that $\dim_K M_{(a+i,b)}< \infty$ and $\dim_K M_{(a-i,b)}< \infty$ for every $i \geq 0$. Consequently, $\dim_k M_{(u,b)}< \infty$ for all $u \in \Z$. If $m \geq 2$, then one can replace $X_n$ by $Y_m$ in \eqref{fin-dim} to obtain that $\dim_K M_{(a, b+j)}< \infty$ and $\dim_K M_{(a,b+j)}< \infty$ for every $j \geq 0$. 
Next, suppose that $m=1$. Then the result follows from Lemma \ref{bi-finite length_lemma}, where we essentially needs the specific choices for $(a_i,b_j)$. Thus in any case, $\dim_K M_{(u,v)}< \infty$ for all $(u,v) \in \Z^2$.
\end{proof}	

\begin{remark}\label{rmk_fin}
The proof of Lemma \ref{bi-finite length_lemma} shows that for any $a \in \mathbb{Z}$, if $\dim_K M_{(a,b_1)} < \infty$ for some $b_1 \ge 0$, then $\dim_K M_{(a,b)}$

\noindent
\begin{minipage}{0.5\textheight}
$< \infty$ for all $b \ge 0$. Likewise, if $\dim_K M_{(a,b_2)} < \infty$ for some $b_2 \le -1$, then $\dim_K M_{(a,b)} < \infty$ for all $b \le -1$. Applying this observation in the inductive step of the proof of Theorem \ref{bi-finite length} for $r=m+n>2$ with either $m=1$ or $n=1$, we obtain the following conclusion: \emph{if $\dim_K M_{(a,b)}<\infty$ for all $(a,b)$ in one of the sectors $\big\{\mathcal{NE}, \mathcal{NW^*}\cup \mathrm{Trun}(\mathcal{N}), \mathbb{Z}^2 \backslash (\mathcal{N}\cup \mathcal{E}), \mathcal{S^*E}\cup \mathrm{Trun}(\mathcal{E})\big\}$, then $\dim_K M_{(u,v)}<\infty$ for all $(u,v)$ in that sector.}
\end{minipage}
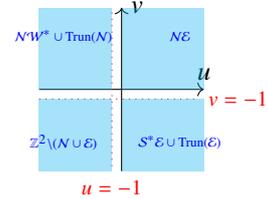
\begin{minipage}{0.225\textheight}
\begin{center}
	\begin{tikzpicture}[scale=0.1275]
	\draw[->] (-8.5,0)--(8.5,0) node[above]{$u$};
	\draw[->] (0,-8.5)--(0,8.5) node[right]{$v$};
	\draw[dotted, red] (-8.5,-1)--(8,-1)node[right]{\footnotesize $v=-1$};
	\draw[dotted, red] (-1,8)--(-1,-8.5) node[below]{\footnotesize $u=-1$};
	\draw[fill=cyan,fill opacity=0.35,draw=none] (0,8.5)--(0,0)--(8.5,0)--(8.5,8.5)--(0,8.5);
	\draw[fill=cyan,fill opacity=0.35,draw=none] (-1,8.5)--(-1,0)--(-8.5,0)--(-8.5,8.5)--(-1,8.5);
	\draw[fill=cyan,fill opacity=0.35,draw=none] (-1,-8.5)--(-1,-1)--(-8.5,-1)--(-8.5,-8.5)--(-1,-8.5);
	\draw[fill=cyan,fill opacity=0.35,draw=none] (0,-8.5)--(0,-1)--(8.5,-1)--(8.5,-8.5)--(0,-8.5);
	\node[blue,draw=none] at (6,5.5) {\tiny $\mathcal{NE}$};
	\node[blue,draw=none] at (-6,5.5) {\tiny $\mathcal{NW^*}\cup \mathrm{Trun}(\mathcal{N})$};
	\node[blue,draw=none] at (-6,-5.5) {\tiny $\mathbb{Z}^2 \backslash (\mathcal{N}\cup \mathcal{E})$};
	\node[blue,draw=none] at (6,-5.5) {\tiny $\mathcal{S^*E}\cup \mathrm{Trun}(\mathcal{E})$};
	\end{tikzpicture}
\end{center}
\end{minipage}
\end{remark}

We need the following lemma to prove the upcoming assertions.

\begin{lemma}\label{poly-two}
Let $f: \N^2 \to \N$ be a bivariate function in $u$ and $v$. Suppose there exist polynomials $P(X,Y), Q(X,Y) \in \QQ[X,Y]$ such that $f(u+1, v)- f(u,v)=P(u,v)$ and $f(u,v+1)-f(u,v)=Q(u,v)$ for all $u,v \gg 0$. Then $f$ is of polynomial-type. 
\end{lemma}
\begin{note*}
Instead, suppose that $f$ is a function $(i) (-\N) \times \N \to \N, (ii) \N \times (-\N) \to \N, (iii) (-\N) \times (-\N) \to \N$ satisfying the above conditions for all $(i) u \ll 0, v \gg 0, (ii) u \gg 0, v \ll 0$ and $(iii) u, v \ll 0$ respectively. Even then using the same line of proof,one can show that $f$ is of polynomial-type. 
\end{note*}
\begin{proof}[Proof of Lemma \ref{poly-two}]
Suppose that there are some $u_0,v_0 \geq 0$ such that $f(u+1, v)- f(u,v)=P(u,v)$ and $f(u,v+1)-f(u,v)=Q(u,v)$ for all $u \geq u_0$ and $v \geq v_0$. Then $f(u,v)=f(u-1, v)+P(u-1, v)= \cdots =f(u_0, v)+\sum_{a=u_0}^{u-1}P(a,v)= f(u_0, v-1)+Q(u_0, v-1)+\sum_{a=u_0}^{u-1}P(a,v)=\cdots= f(u_0,v_0)+\sum_{b=v_0}^{v-1}Q(u_0,b)+\sum_{a=u_0}^{u-1}P(a,v)$ for all $u \geq u_0$ and $v \geq v_0$. Note that the last two terms are polynomial functions. The result follows. 
\end{proof}
\begin{remark}
Note that $X$-degree of the polynomial $\sum_{a=u_0}^{X-1}P(a,Y)$ is $\xdeg P[X,Y]+1$  and $Y$-degree of the polynomial $\sum_{b=v_0}^{Y-1}Q(u_0,b)$ is $\ydeg Q[X,Y]+1$, see \cite[Lemma 4.1.2]{BH}. Thus 
\begin{align*}
&\xdeg f=\max\{\xdeg P[X,Y]+1, \xdeg Q[X,Y]\} \\
\mbox{and} \quad & \ydeg f=\max\{\ydeg P[X,Y], \ydeg Q[X,Y]+1\}.
\end{align*} 
\end{remark}

We also have the following result from \cite[Proposition 2.3.6]{GCN}.
\begin{proposition}
Let $f,g: \N^r \to \Z$ be functions, $\gamma_1, \ldots, \gamma_r \in \N^r$ be linearly independent vectors and $\underline{\beta} \in \N^r$, such that for all $\underline{\alpha} \in \N^r$ and for some $i=1, \ldots, r$ it holds \[f(\underline{\alpha})-f(\underline{\alpha}-\gamma_i)=g(\underline{\alpha}).\]
If $g$ is a quasi polynomial on $\underline{\beta}, \gamma_1, \ldots, \gamma_r$ of polynomial of degree $d$, then $f$ is also a quasi polynomial on $\underline{\beta}, \gamma_1, \ldots, \gamma_r$ of polynomial of degree $d+1$.
\end{proposition}

\begin{theorem}\label{bi-poly}
Let $n,m\geq 1$. Let $M = \bigoplus_{(u,v) \in \Z^2}M_{(u,v)}$ be a bigraded holonomic generalized Eulerian $A_{n,m}(K)$-module. Then there exist polynomials $P^a_M(X,Y)\in \mathbb{Q}[X,Y]$ for $a=1,2,3,4$ such that 

\noindent
\begin{minipage}{0.636\textheight}
	\begin{enumerate}[\rm (i)]
		\item $P^1_M(u,v) = \dim_K M_{(u,v)}$  for all  $u,v \gg 0,$ if $\dim_K M_{(u_1,v_1)})$ is finite for some $(u_1,v_1)\in \mathcal{NE}$.
		\item $P^2_M(u,v) = \dim_K M_{(u,v)} $   for all  $u \ll 0,  v \gg 0,$  if $\dim_K M_{(u_1,v_1)})$ is finite for some $(u_1,v_1)\in \mathcal{NW^*}$
		\item $P^3_M(u,v) = \dim_K M_{(u,v)}$   for all  $  u,v \ll 0,$  if $\dim_K M_{(u_1,v_1)})$ is finite for some $(u_1,v_1)\in \mathcal{S^*W^*}$.
		\item $P^4_M(u,v) = \dim_K M_{(u,v)}$  for all  $  u \gg 0, v \ll 0,$  if $\dim_K M_{(u_1,v_1)})$ is finite for some $(u_1,v_1)\in \mathcal{N^*E}$.
	\end{enumerate}
\end{minipage}
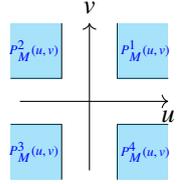
\begin{minipage}{0.11\textheight}
\begin{center}
	\begin{tikzpicture}[scale=0.122]
	\draw[->] (-7.5,0)--(8.5,0) node[below]{$u$};
	\draw[->] (0,-7.5)--(0,8.5) node[above]{$v$};
	\draw[-](3,8.5)--(3,2.5)-- (8.5,2.5);
	\draw[-](-3,8.5)--(-3,2.5)-- (-8.5,2.5);
	\draw[-](-3,-8.5)--(-3,-2.5)-- (-8.5,-2.5);
	\draw[-](3,-8.5)--(3,-2.5)-- (8.5,-2.5);
	\draw[fill=cyan,fill opacity=0.35,draw=none] (3,8.5)--(3,2.5)--(8.5,2.5)--(8.5,8.5)--(3,8.5);
	\draw[fill=cyan,fill opacity=0.35,draw=none] (-3,8.5)--(-3,2.5)--(-8.5,2.5)--(-8.5,8.5)--(-3,8.5);
	\draw[fill=cyan,fill opacity=0.35,draw=none] (-3,-8.5)--(-3,-2.5)--(-8.5,-2.5)--(-8.5,-8.5)--(-3,-8.5);
	\draw[fill=cyan,fill opacity=0.35,draw=none] (3,-8.5)--(3,-2.5)--(8.5,-2.5)--(8.5,-8.5)--(3,-8.5);
	\node[blue,draw=none] at (6,5.5) {\tiny $P^1_M(u,v)$};
	\node[blue,draw=none] at (-6,5.5) {\tiny $P^2_M(u,v)$};
	\node[blue,draw=none] at (-6,-5.5) {\tiny $P^3_M(u,v)$};
	\node[blue,draw=none] at (6,-5.5) {\tiny $P^4_M(u,v)$};
	\end{tikzpicture}
\end{center}
\end{minipage}
\end{theorem}

\begin{proof}
Applying Remark \ref{rmk_fin} to (i)-(iv), we obtain that $\dim_K M_{(u,v)}$ is finite for all $(u,v)$ in the respective regions.

\vspace{0.15cm}
We use induction on $r=n+m$. If $r=1$, then either $n=1,m=0$ or $n=0, m=1$. In either case (equivalent to the standard graded case), the result follows from \cite[Theorem 11.1]{TP2}.

If $r=2$, then $(i)~n=2, m=0, (ii)~n=1,m=1$, or $(iii)~n=0, m=2$. For the two cases $(i)$ and $(iii)$, the result follows from \cite[Theorem 11.1]{TP2}. Now let us assume that $m=1=n$. Consider the exact sequence 
	\begin{equation}\label{tm1a}
		0 \rt H_1(\partial_1; M)_{(u,v)} \rt M_{(u+1,v)} \overset{\cdot \partial_1}{\lrt} M_{(u,v)} \rt H_0(\partial_1; M)_{(u,v)}\rt 0.
	\end{equation}
	For $i=0,1$, as $H_i(\partial^X_1, M)_{(u,v)}=0$ for all $u \neq -1$ and any $v \in \Z$ by Proposition \ref{bigrd-genEur-partial-con} so from the exact sequence \eqref{tm1a} it follows that
	\begin{align*}
	&M_{(u,v)} \cong M_{(-1,v)} \quad \text{for all } u \leq -1 \text{ and any } v\in \Z\\
\text{and} \quad	&M_{(u,v)} \cong M_{(0,v)} \quad \hspace{0.2cm} \text{for all } u \geq 0 \text{ and any } v\in \Z.
	\end{align*}
	Now consider the exact sequence 
	\begin{equation}\label{tm1b}
		0 \rt H_1(\delta_1; M)_{(u,v)} \rt M_{(u,v+1)} \overset{\cdot \delta_1}{\lrt} M_{(u,v)} \rt H_0(\delta_1; M)_{(u,v)}\rt 0.
	\end{equation}
	For $i=0,1$, as $H_i(\delta_1, M)_{(u,v)}=0$ for all $v \neq -1$ and any $u \in \Z$ by Proposition \ref{bigrd-genEur-partial-con} so from the exact sequence \eqref{tm1b} we get 
	\begin{align*}
		& M_{(u,v)} \cong M_{(u,-1)} \quad \text{for all } v \leq -1 \text{ and any } u\in \Z\\
	\text{and} \quad	& M_{(u,v)} \cong M_{(u,0)} \quad \hspace{0.2cm} \text{for all } v \geq 0 \text{ and any } u\in \Z.
	\end{align*}
Consequently, 
\[ M_{(u,v)}\cong \begin{cases} 
      M_{(0,0)} & \mbox{if } u, v \geq 0 \\
      M_{(-1,0)} & \mbox{if } u\leq -1 \mbox{ and } v \geq 0 \\
      M_{(-1,-1)} & \mbox{if } u, v \leq -1 \\
      M_{(0,-1)} & \mbox{if } u\geq 0 \mbox{ and } v \leq -1 
   \end{cases}
\]
Take $P^1_M(u,v)= M_{(0,0)}, P^2_M(u,v)=M_{(-1,0)}, P^3_M(u,v)=M_{(-1,-1)}, P^4_M(u,v)=M_{(0,-1)}$ for the result.

Next, suppose that $r > 2$ and the result holds true for $<r$. The case when either $m=0$ or $n=0$ follows from \cite[Theorem 11.1]{TP2}. So we can assume that $m, n \geq 1$. Consider the exact sequences
	\begin{equation}
		\begin{split}
			0 \rt H_1(\partial_n; M)_{(u,v)} \rt M_{(u+1,v)} \overset{\cdot \partial_n}{\lrt} M_{(u,v)} \rt H_0(\partial_n; M)_{(u,v)}\rt 0.\\
			0 \rt H_1(\delta_m; M)_{(u,v)} \rt M_{(u,v+1)} \overset{\cdot \delta_m}{\lrt} M_{(u,v)} \rt H_0(\delta_m; M)_{(u,v)}\rt 0.
		\end{split}
	\end{equation}

	For $j=0, 1$, we have $H_j(\partial_n; M)(-1,0)$ is a bigraded generalized Eulerian $A_{n-1,m}(K)$-module and $H_j(\delta_m; M)(0,-1)$ is a bigraded generalized Eulerian $A_{n,m-1}(K)$-module. For every pair $(u,v) \in \Z^2$, since $\dim_K M_{(u,v)}< \infty$, we get that $\dim_K H_j(\partial_n, M)_{(u,v)}< \infty$ and $\dim_K H_j(\delta_m, M)_{(u,v)}< \infty$ for $j=0,1$. Therefore, by induction hypothesis 
	\begin{align*}
	\dim_K M_{(u+1, v)}-\dim_K M_{(u,v)}&=\dim_K H_1(\partial_1, M)_{(u,v)}- \dim_K H_0(\partial_1, M)_{(u,v)}\\
	&=P^1_{H^X_1}(u,v)-P^1_{H^X_0}(u,v) \mbox{ for all } u,v \gg 0,\\
 \mbox{and } \quad	\dim_K M_{(u, v+1)}-\dim_K M_{(u,v)}&=\dim_K H_1(\delta_1, M)_{(u,v)}- \dim_K H_0(\delta_1, M)_{(u,v)}\\
		&=P^1_{H^Y_1}(u,v)-P^1_{H^Y_0}(u,v) \mbox{ for all } u,v \gg 0.
	\end{align*}
Hence by Lemma \ref{poly-two}, $\dim_K M_{(u,v)}$ is given by a polynomial in two variables for all $u,v \gg 0$. Using similar arguments and induction hypothesis we further obtain that $\dim_K M_{(u,v)}$ are also given by polynomials for $u \ll 0, v \gg 0$; $u , v \ll0$ and $u \gg 0, v \ll 0$. This establishes the result.
\end{proof}

\begin{example} 
Let $T = K[X_1,\ldots, X_n, Y_1, \ldots, Y_m] = \bigoplus_{{(u,v)} \in \N^2} T_{(u,v)}$. Then it is well-known that
\begin{enumerate}[\rm(1)]
	\item For all $(u,v) \in \N^2$,
	\[
	\dim_K T_{(u,v)} = \binom{u + n -1}{n-1}\binom{v + m -1}{m-1}. 	\]
	\item Recall that 
	\[E:=E_T(K)=(X_1^{-1}\cdots X_n^{-1}Y_1^{-1}\cdots Y_m^{-1})K[X_1^{-1}, \ldots, X_n^{-1}, Y_1^{-1}, \ldots, Y_m^{-1}],\]
	where $E_T(K)$ denotes the injective hull of $K$ over $T$. Notice $E_T(K)$ has a natural bigrading induced from the ring $R$. Thus
	\begin{align*}
	\dim_K E_{(u,v)} = \binom{(n-u) + n -1}{n-1}\binom{(m-v) + m -1}{m-1} \quad \text{ for all } (u,v)\leq(-n,-m).
	\end{align*}
\end{enumerate} 
\end{example}

\begin{note*} 
For sets of arbitrary integers $\{i_1,\ldots,i_n\}$ and $\{j_1,\ldots, j_m\}$, we have 
	\begin{align*}
	&\e^X_n \cdot X_1^{i_1} \cdots X_n^{i_n}= (i_1+ \cdots+i_n)X_1^{i_1} \cdots X_n^{i_n} \\
	\text{and} \quad&  \e^Y_m \cdot Y_1^{j_1} \cdots Y_m^{j_m}= (j_m+ \cdots+j_m) Y_1^{j_1} \cdots Y_m^{j_m}.
	\end{align*}
	Consequently, both $T$ and $E(n,m)$ are bigraded Eulerian $A_{n,m}(K)$-module. Observe that their nonzero bigraded pieces can be represented by the regions $\mathcal{NE}$ and $\mathcal{S^*W^*}$ (see Notations \ref{notations}), respectively. Thus they satisfy the rigidity property stated in Theorem \ref{bi-rigid}. 
\end{note*}

We need the following lemma. 
\begin{lemma}\label{X_inv_D_act}
	Let $R=K[X_1, \ldots, X_n]$ be a polynomial ring and $A_n(K):=R \langle \partial_1, \ldots, \partial_n\rangle$ be the $n^{th}$ Weyl algebra over $K$, where $\partial_i=\partial/\partial X_i$ for $i=1, \ldots,n$. Then for $c, a_1, \ldots, a_n, b_1, \ldots, b_n \geq 0$,
	\begin{enumerate}[\rm i)]
		\item $\partial_i^c(X_1^{-a_1} \cdots X_n^{-a_n})=(-1)^c\frac{(a_i+c-1)!}{(a_i-1)!}X_i^{-a_i-c} \cdot \Big(\prod_{\substack{j=1\\j \neq i}}X_j^{-a_j}\Big)$,
		\item $\partial_i^c(X_1^{b_1} \cdots X_n^{b_n})=\frac{b_i!}{(b_i-c)!}X_i^{b_i-c} \cdot \Big(\prod_{\substack{j=1\\j \neq i}}X_j^{b_j}\Big)$.
	\end{enumerate}
\end{lemma}

\begin{proof}
For any $f,g \in R=K[X_1, \ldots, X_n]$ with $f \neq 0$, we have \begin{equation}\label{inv_rel}
\partial_i\left(\frac{g}{f^k}\right)=\frac{\partial_i(g)f-k\partial_i(f)g}{f^{k+1}}=\partial_i(g)f^{-k}-kg\partial_i(f)f^{-k-1}.
\end{equation}
Clearly $\partial_i(X_j^{-d})=\partial_i(1)X_j^{-d}-d\partial_i(X_j)X_j^{-d-1}=-dX_i^{-d-1}$ if $j=i$, and $\partial_i(X_j^{-d})=0$ if $j \neq i$. Now $\partial_i^c(X_i^{-d})=\partial_i^{c-1} \cdot \partial_i(X_i^{-d})=\partial_i^{c-1} \left(-dX_i^{-d-1}\right)=-d\partial_i^{c-2} \cdot \partial_i\left(X_i^{-d-1}\right)=(-d)(-d-1)\partial_i^{c-2}\left(X_i^{-d-2}\right)=\cdots=(-d)(-d-1) \cdots (-d-c+1)X_i^{-d-c}=(-1)^c\frac{(d+c-1)!}{(d-1)!}X_i^{-d-c}$ for any $c, d \geq 1$. Using \eqref{inv_rel}, and the fact $\partial_i(fg)=f\partial_i(g)+\partial_i(f) g$ for all $f,g \in R$ repeatably, we get that
\[\partial_i^c(X_1^{-a_1} \cdots X_n^{-a_n})= \partial_i^c(X_i^{-a_i}) \cdot \Big(\prod_{\substack{j=1\\j \neq i}}X_j^{-a_j}\Big)=(-1)^c\frac{(a_i+c-1)!}{(a_i-1)!}X_i^{-a_i-c} \cdot \Big(\prod_{\substack{j=1\\j \neq i}}X_j^{-a_j}\Big).\]
Again for any $c,d \geq 1$, we have $\partial_i^c(X_i^{d})=\partial_i^{c-1} \cdot \partial_i(X_i^{d})=\partial_i^{c-1} \left(dX_i^{d-1}\right)=d\partial_i^{c-2} \cdot \partial_i\left(X_i^{d-1}\right)=d(d-1)\partial_i^{c-2}\left(X_i^{d-2}\right)=\cdots=d(d-1) \cdots (d-c+1)X_i^{d-c}=\frac{d!}{(d-c)!}X_i^{d-c}$ if $c \leq d$, else $\partial_i^c(X_i^{d})=0$. Therefore,
\[\partial_i^c(X_1^{b_1} \cdots X_n^{b_n})= \partial_i^c(X_i^{b_i}) \cdot \Big(\prod_{\substack{j=1\\j \neq i}}X_j^{b_j}\Big)=\frac{b_i!}{(b_i-c)!}X_i^{b_i-c} \cdot \Big(\prod_{\substack{j=1\\j \neq i}}X_j^{b_j}\Big).\]
\end{proof}

We now recall the following well-known result. 

\begin{lemma}\label{quo_poly_rel}
Take two subsets $W \subseteq \{1, \ldots, n\}$ and $V \subseteq \{1, \ldots, m\}$.  Consider the left ideal $\mathfrak{b}$ in $\D:=A_{(n, m)}(K)$ generated by $\{X_i, Y_j \mid i \in W, j \in V\}$ and $\{\partial_{i'}, \delta_{j'} \mid i' \notin W, j' \notin V\}$. Then 	
\[\frac{\D}{\D \mathfrak{b}}=K\big\{X_{i'}^{a_{i'}}Y_{j'}^{b_{j'}}\partial_i^{c_i}\delta_j^{d_j} \mid \mbox{ for some }i \in W, j \in V,i' \notin W, j' \notin V\big\}\]
and the map
\[\varphi: \frac{\mathcal{D}}{\mathcal{D}\mathfrak{b}}\to  F:=\left(\prod_{i \in W} X_i^{-1}\right)\left(\prod_{j \in V} Y_j^{-1}\right)K\left[X_{i'}, Y_{j'},X_i^{-1}, Y_j^{-1} \mid i \in W, j \in V, i' \notin W, j' \notin V\right]\]
define by 
\[X_{i'}^{a_{i'}}Y_{j'}^{b_{j'}}\partial_i^{c_i}\delta_j^{d_j} \mapsto (-1)^{c_i+d_j}c_i!\cdot d_j!\cdot X_{i'}^{a_{i'}}Y_{j'}^{b_{j'}}X_i^{-c_i-1}Y_j^{-d_j-1}\]
is an isomorphism of 
left $\D$-modules, where the action of $\D$ on $F$ is defined using Lemma \ref{X_inv_D_act}.
\end{lemma}

\begin{remark}\label{rmk_quo_poly_rel}
Note that in Lemma \ref{quo_poly_rel}, $\varphi$ is $\D$-linear, $\overline{1}$ in $\frac{\D}{\D \mathfrak{b}}$ maps to $\left(\prod_{i \in W} X_i^{-1}\right) \left(\prod_{j \in V} Y_j^{-1}\right)$ and $\bideg \left(\prod_{i \in W} X_i^{-1}\right)$ $ \cdot \left(\prod_{j \in V} Y_j^{-1}\right)=(|W|, |V|)$. So there exists a bigraded isomorphism
\[\frac{\mathcal{D}}{\mathcal{D}\mathfrak{b}}\left((-|W|, -|V|)\right) \overset{\varphi}{\cong} \left(\prod_{i \in W} X_i^{-1}\right)\left(\prod_{j \in V} Y_j^{-1}\right)K\left[X_i^{-1}, Y_j^{-1}, X_{i'}, Y_{j'} \mid i \in W, j \in V, i' \notin W, j' \notin V\right]\]
of left $A_{(n,m)}(K)$-modules. 
\end{remark}

\begin{proposition} \label{msudan}
Let $n,m \geq 2$. Let $M = \bigoplus_{(u,v) \in \Z^2}M_{(u,v)}$ be a bigraded generalized Eulerian holonomic $A_{n,m}(K)$-module. Suppose 
$\dim_K M_{(0,0)} $ is finite and that for each of the regions $\mathrm{Trun}(\mathcal{N}), \mathrm{Trun}(\mathcal{W}^*)$ {\rm(or,} $ \mathrm{Trun}(\mathcal{S}^*)${\rm)} and $\mathrm{Trun}(\mathcal{E})$, there exists $(u_1,v_1)$ with $M_{(u_1,v_1)}=0$, but $M \neq 0$. Then $\dim_K M_{(u,v)} $ is finite and non-zero in some of the four regions $\mathcal{NE}$,  $\mathcal{NW^*}$, $\mathcal{S^*W^*}$ and $\mathcal{S^*E}$. Moreover, $\dim_K M_{(u,v)}$ in these regions is given in the following table:

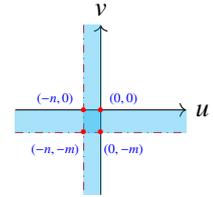
\begin{minipage}{0.55\textheight}	
	\vspace{-0.2cm}
	\begin{table}[H]
		\centering
		\begin{tabular}{|l|l|}
			\hline
			\textbf{Blocks}  &   $\dim_K M_{(u,v)}$   \\
			\hline
			$\mathcal{NE} $ &   $\dim_K M_{(0,0)} \cdot \binom{u + n -1}{n-1}\binom{v + m -1}{m-1}$  \\
			\hline
			$\mathcal{NW^*}$ &  $\dim_K M_{(-n,0)} \cdot \binom{(n-u) + n -1}{n-1}\binom{v + m -1}{m-1}$  \\
			\hline  
			$\mathcal{S^*W^*}$   &   $\dim_K M_{(-n,-m)} \cdot \binom{(n-u) + n -1}{n-1}\binom{(m-v) + m -1}{m-1} $ \\
			\hline    
			$\mathcal{S^*E}$  &    $\dim_K M_{(0,-m)} \cdot \binom{u + n -1}{n-1}\binom{(m-v) + m -1}{m-1}$\\
			\hline
		\end{tabular}
	\end{table}
\end{minipage}
\begin{minipage}{0.125\textheight}
\begin{center}
	\begin{tikzpicture}[scale=0.15]
	\draw[->] (-7.5,0)--(7.5,0) node[right]{$u$};
	\draw[->] (0,-7.5)--(0,7.5) node[above]{$v$};
	\draw[dashdotted, red] (-1.5,7.5)--(-1.5,-7.5);
	\draw[dashdotted, red] (-7.5,-2)--(7.5, -2);
	\draw[fill=cyan,fill opacity=0.35,draw=none] (-1.5,-7.5)--(0,-7.5)--(0,7.5)--(-1.5,7.5);
	\draw[fill=cyan,fill opacity=0.35,draw=none] (7.5,-2)--(7.5,0)--(-7.5,0)--(-7.5,-2)--(7.5,-2);
	\node[blue,draw=none] at (2,-3.5) {\tiny $(0,-m)$};
	\node[red] at (0,-2){\tiny$\bullet$};
	\node[blue,draw=none] at (2,1) {\tiny $(0,0)$};
	\node[red] at (0,0){\tiny$\bullet$};
	\node[blue,draw=none] at (-4,1) {\tiny $(-n,0)$};
	\node[red] at (-1.5,0){\tiny$\bullet$};
	\node[blue,draw=none] at (-4,-3.5) {\tiny $(-n,-m)$};
	\node[red] at (-1.5,-2){\tiny$\bullet$};
	\end{tikzpicture}
\end{center}
\end{minipage}

\vspace{0.15cm}
\noindent
and $M_{(u,v)} = 0$ otherwise.

\end{proposition}
\begin{proof}
	From Theorem \ref{bi-finite length} we have $\dim_K M_{(u,v)} < \infty$ for all $(u,v) \in \Z^2$. Besides, by Theorem \ref{bi-rigid}, we get $M_{(u,v)} = 0$ for all $(u,v)$ in the shaded region in the above figure. As $M\neq 0$ so $M_{(u,v)} \neq 0$ in some of 
$\mathcal{NE}$,  $\mathcal{NW^*}$, $\mathcal{S^*W^*}$ and $\mathcal{S^*E}$.

	For convenience, set $\D = A_{n,m}(K)$ and consider $T = K[X_1,\ldots, X_n, Y_1, \ldots, Y_m]$ with its unique maximal homogeneous  ideal $\n = (X_1,\ldots, X_n, Y_1, \ldots, Y_m)$. Further let $E = \D/\D\n$ denote the injective hull of $K = T/\n$ as a $T$-module.
	
	For the block $\mathcal{NE}$, Let $p = \dim_k M_{(0,0)}$. Suppose that $\{u_1,\ldots, u_p\}$ is a $K$-basis of $M_{(0,0)}$. We consider the map $\psi \colon \D^p \rt M$ which sends $e_i$ to $u_i$ for $i=1, \ldots, p$. Set $\mathfrak{b}:=\D\underline{\partial}+\D\underline{\delta}$, where $\underline{\partial}:=\partial_1,\ldots,\partial_n$ and $\underline{\delta}:=\delta_1, \ldots, \delta_m$. Notice
	$\psi (\mathfrak{b}) = 0$, as $M_{(-1,0)}=0=M_{(0,-1)}$. So $\psi$ factors through a map $\ov{\psi} \colon T^p \rt M$, as $T\cong \D/\mathfrak{b}$. Thus we have an exact sequence 
	\[
	0 \rt K \rt T^p \xrightarrow{\ov{\psi}} M \rt C \rt 0 \]
	of bigraded generalized Eulerian holonomic $\D$-modules, $\text{where} \ K = \ker \ov{\psi} \ \text{and} \ C = \coker \ov{\psi}$. Note that by construction, $K_{(u,v)} = C_{(u,v)} = 0$ for  all $(u,v) \in \Z^2 \backslash [0,\infty) \times [0,\infty)$, as $T$ and $M$ have this property. Again by the construction of the map, $(T_{(0,0)})^p \cong M_{(0,0)}$. So $K_{(0,0)} = 0= C_{(0,0)}$. From Theorem \ref{bi-rigid}(I) it follows that $K_{(u,v)} =0= C_{(u,v)}$ for all $(u,v) \geq (0,0)$, that is, for all $(u,v) \in [0,\infty) \times [0,\infty)$. Thus $K = C = 0$ and hence $M\cong T^p$. The result follows.
	
	\vspace{0.15cm}
	One can obtain the above result for the blocks $\mathcal{NW^*}, \mathcal{S^*W^*}$ and $\mathcal{S^*E}$ by considering the ideal $\mathfrak{b}$ in $\D$ to be $(X_1, \ldots,X_n, \delta_1,\ldots,\delta_m)$, $(X_1, \ldots, X_n, Y_1, \ldots, Y_m)$ and $(\partial_1, \ldots, \partial_n, Y_1, \ldots, Y_m),$ respectively (see Lemma \ref{quo_poly_rel} and Remark \ref{rmk_quo_poly_rel}) and using the same line of proof as above. This establishes the assertions.
\end{proof}

\begin{remark}\label{fin_iso}
The above result shows that if the nonzero components of a bigraded generalized Eulerian holonomic $A_{n,m}(K)$-module $M$ is represented by the blocks $\mathcal{NE}, \mathcal{NW^*}, \mathcal{S^*W^*}, \mathcal{S^*E}$ in Theorem \ref{bi-rigid} respectively, then $M \cong (\D/\mathfrak{b})^p$, where $\mathfrak{b}$ denotes the ideal $(X_1, \ldots,X_n,Y_1, \ldots, Y_m)$, $(X_1, \ldots,X_n, \delta_1,\ldots,\delta_m)$, $(X_1, \ldots, X_n, Y_1, \ldots, Y_m)$, $(\partial_1, \ldots, \partial_n, Y_1, \ldots, Y_m)$ in $\D$, and `$p$' is the $K$-vector space dimension of the pieces $M_{(0,0)}, M_{(-n,0)},M_{(-n,-m)}, M_{(0,-m)}$, respectively. 
\end{remark}

\section{Application to local cohmology supported on bihomogeneous ideals}
In this section, we restrict our attention to local cohomology modules. Our goal is to generalize the results proven by second author in \cite{TP2}. Throughout this section, we assume that $C$ is a regular ring. Our proofs use similar steps as in \cite{TP2} under the following framework.

\s\label{sa_application}{\bf Setup}: Let $C$ be a regular ring containing a field $K$ of characteristic zero and let $R=C[X_1, \ldots, X_n, Y_1, \ldots, Y_m]$ be a polynomial ring over $C$ with $\bideg c=(0,0)$ for all $c \in C$, $\bideg X_i=(1,0)$ and $\deg Y_j=(0,1)$ for $i=1, \ldots, n$ and $j=1, \ldots, m$. Let $\mathcal{T}$ be a bigraded Lyubeznik functor on ${}^*\Mod(R)$. Set $M:=\TT(R)=\bigoplus_{(u,v)\in \Z^2}M_{(u,v)}$.
	
\s{\bf Growth of Bass numbers}.\\ 
We will need the following Lemma from \cite[1.4]{Lyu1}.
\begin{lemma}\label{lyu-lemma}
Let $B$ be a Noetherian ring and $N$ be a (need not be finitely generated) $B$-module. Let $P$ be a prime ideal in $B$. If $(H^j_P(N))_P$ is injective for all $j \geq 0$ then
\[\mu_j(P,N) = \mu_0(P,H^j_P(N)) \quad \mbox{ for all } j \geq 0.\]
\end{lemma}

\begin{theorem}[With hypotheses as in \ref{sa_application}]\label{comp-inj}
Let $P$ be a prime ideal in $C$. Set $N:=M_{(u,v)}$ for some $(u,v) \in \Z^2$. Then $H_P^i(N)_P$ is injective for all $i \geq 0$.	
\end{theorem}

\begin{proof}
Note that $H^i_{PR} \circ \mathcal{T}$ is a bigraded Lyubeznik functor and \[H^j_{PR}(\mathcal{T}(R))_{(u,v)}=H^j_{P}(\mathcal{T}(R)_{(u,v)})=H^j_{P}(N)\] for each $j \geq 0$ and all $(u,v) \in \Z^2$. Clearly, either $H^j_{P}(N)_P = 0$ or $P$ is a minimal prime of $H^j_{P}(N)$. Let $H^j_{P}(N)_P \neq 0$. Set $B=\widehat{C_P}$ and put $S=B[X_1, \ldots, X_n, Y_1, \ldots, Y_m]$. As all actions are on degree zero elements so by a similar argument as in \cite[2.6]{TP2} for the graded case we get that $G(-)=B \otimes_C H^j_{PR} \circ \mathcal{T}(-)=S \otimes_R H^j_{PR} \circ \mathcal{T}(-)=H^j_{PS} \circ \mathcal{T}(-)$ is a bigraded Lyubeznik functor on $B \otimes_C R= \widehat{C_P} \otimes_{C_P} C_P \otimes_C C[X_1, \ldots, X_n, Y_1, \ldots, Y_m]=S$. Since $C$ is a regular ring, by {\it Cohen-structure theorem} we have $B=K[[Z_1, \ldots, Z_g]]$, where $K=k(P)$, the residue field of $C_P$ and $g=\hgt P$. Let $\Lambda$ be the $K$-linear differential operators on $B$ and $A_{n,m}(\Lambda)$ be the standard bigraded $(n,m)^{th}$-Weyl algebra over $\Lambda$. Then by Theorem \ref{Mul-LC} and \cite[Theorem 4.2]{TP2} it follows that $L:=G(S)=S \otimes_R H^j_{PR}(\mathcal{T}(R)) =H^j_{PS} \circ \mathcal{T}(S)$ is a bigraded generalized Eulerian $A_{n,m}(\Lambda)$-module. Hence from \cite[2.12]{TP2} we get $L_{(u,v)}=H^j_{P}(N)_P \neq 0$. As $P$ is a minimal prime of $H^j_{P}(N)$ so $L_{(u,v)}$ is supported only at $\n=(Z_1, \ldots, Z_g)$, the unique maximal ideal of $B$. Since $\Lambda$ is contained in the degree zero component of $A_{n,m}(\Lambda)$, we get $L_{(u,v)}$ is also a $\Lambda$-module. Thus by \cite[Theorem 2.4]{Lyu1}, $L_{(u,v)} \cong E_B(K)^{\alpha}$, where $E_B(K)$ is the injective hull of $K$ as a $B$-module and $\alpha$ is some ordinal (possibly infinite). Since $E_B(K) \cong E_{C_P}(k(P)) \cong E_C(C/P)$ as $A$-modules, it follows that $H^j_{P}(N)_P$ is an injective $C$-module.
\end{proof}

\begin{remark}\label{rmk_Bass}
From Theorem \ref{comp-inj} and Lemma \ref{lyu-lemma} it follows that $\mu_j(P,M_{(u,v)}) = \mu_0(P,H^j_P\big(M_{(u,v)}\big)$ for all $j \geq 0$ and all $(u,v)\in \Z^2$.
\end{remark}

We need the following results.
\begin{proposition}\label{Koszul-local}
Let $E$ be a bigraded generalized Eulerian holonomic  $A_{n,m}(\Lambda)$-module. Then $H_j(Y_g, E)$ are bigraded generalized Eulerian holonomic $A_{n,m}(\Lambda')$-modules for $j = 0, 1$, where $\Lambda'=C' \langle \delta_1, \ldots, \delta_{g-1}\rangle$ with $C'=K[[Z_1, \ldots, Z_{g-1}]]$. 
\end{proposition}
\begin{proof}
Since $H_1(Y_g, E)=\{e \in E \mid Y_ge=0\},~H_0(Y_g, E)=E/Y_gE$ and the map $E \xrightarrow{Y_g} E$ is $A_{n,m}(\Lambda')$-linear, from Proposition \ref{exact} it follows that $H_j(Y_g, E)$ are bigraded generalized Eulerian $A_{n,m}(\Lambda')$-modules for $j = 0, 1$. Moreover, by \cite[Theorem 4.25]{TP2}, $H_j(Y_g, E)$ are holonomic $A_{n,m}(\Lambda')$-modules for $j = 0, 1$.
\end{proof}
Using induction to the above theorem we get the following statement. 
\begin{corollary}\label{full-koszul-hom}
Let $E$ be a bigraded generalized Eulerian holonomic  $A_{n,m}(\Lambda)$-module. Then for every $\nu \geq 0$, the Koszul homology modules $H_\nu(Y_1,\ldots, Y_g; E)$ are bigraded generalized Eulerian holonomic $A_{n,m}(K)$-modules.
\end{corollary}

Next we deduce \cite[Theorem 9.2]{TP2} for the bigraded case.

\begin{theorem}[With hypotheses as in \ref{sa_application} and notations as in Subsection \ref{notations}]	\label{bi-bass} Let $P$ be a prime ideal in $C$. Fix $j \geq 0$. EXACTLY one of the following holds:
	\begin{enumerate}[\rm(i)]
		\item
		$\mu_j(P, M_{(u,v)})$ is infinite for all $(u,v) \in \Z^2$.
		\item $\mu_j(P,M_{(u,v)})$ is finite for all $(u,v)$ in a proper subset of $\big\{\mathcal{NE}, \mathcal{NW^*}\cup \mathrm{Trun}(\mathcal{N}), \mathbb{Z}^2 \backslash (\mathcal{N}\cup \mathcal{E}), \mathcal{S^*E}\cup \mathrm{Trun}(\mathcal{E})\big\}$ and infinite otherwise.
		
		\item
		$\mu_j(P, M_{(u,v)})$ is finite for all $(u,v) \in \Z^2$. In this case EXACTLY one of the following holds:
		\begin{enumerate}[\rm (a)]
	\item $\mu_j(P, M_{(u,v)}) = 0$ for all $(u,v) \in \Z^2$.
	\item $\mu_j(P, M_{(u,v)}) \neq 0$ for all $(u,v)$ in the union of some regions in 
$\big\{\mathcal{NE}, \mathcal{NW^*}, \mathcal{S^*W^*}, \mathcal{S^*E}, \mathcal{C}, \mathcal{N}, \mathcal{W^*}, \mathcal{S^*}, \mathcal{E}\big\}$ and $\mu_j(P, M_{(u,v)}) = 0$ for all $(u,v)$ in the complementary region.	
		\end{enumerate}
	\end{enumerate}
\end{theorem}

\begin{proof}
By Remark \ref{rmk_Bass}, $\mu_j(P,M_{(u,v)})=\mu_0(P,H_P^j(M_{(u,v)}))$ $\mbox{for all } (u,v) \in \Z^2$. Then from the proof of Theorem \ref{comp-inj} we get that $\mu_j(P,M_{(u,v)})=\alpha_{(u,v)}$, as $H^j_{P}(M_{(u,v)})_P=L_{(u,v)} \cong E_B(k)^{\alpha_{(u,v)}}$, where $B=\widehat{C_P}$. By Corollary \ref{full-koszul-hom}, $V=H_g(Y_1,\ldots, Y_g; L)$ is a bigraded generalized Eulerian holonomic $A_{n,m}(K)$-module with 
\begin{align*}
V_{(u,v)}&=H_g({\bf Y}; L)_{(u,v)}=H_g({\bf Y}; L_{(u,v)})=H_g({\bf Y}; E_B(K)^{\alpha_{(u,v)}})=H_g({\bf Y}; E_B(K))^{\alpha_{(u,v)}}=K^{\alpha_{(u,v)}},
\end{align*}
where ${\bf Y}:=Y_1,\ldots, Y_g$ and the last equality holds from \cite[2.11]{TP2}. Notice $\dim_K V_{(u,v)}= \alpha_{(u,v)}=\mu_j(P,M_{(u,v)})$. Since $V$ is a bigraded generalized Eulerian, holonomic $A_{n,m}(K)$-module, finitesness as in (ii) and (iii) follows from Theorem \ref{bi-finite length}. Let $\mu_j(P,M_{(u_0,v_0)})=\alpha_{(u_0, v_0)}\neq 0$ for some $(u_0,v_0) \in \Z^2$. Then $V_{(u_0,v_0)}=K^{\alpha_{(u_0, v_0)}} \neq 0$. The result follows by applying Theorems \ref{bi-rigid} and \ref{bi-rigid2} to $V$.
\end{proof}

\begin{theorem}[With hypotheses as in \ref{sa_application}]
Let $n,m\geq 1$. Let $P$ be a prime ideal in $C$. Fix $j \geq 0$. Then there exist polynomials $P^a_M(X,Y)\in \mathbb{Q}[X,Y]$ for $a=1,2,3,4$ such that 
	\begin{enumerate}[\rm(i)]
		\item $P^1_M(u,v) = \mu_j(P,M_{(u,v)})$  for all  $u,v \gg 0,$ if $\mu_j(P,M_{(u_1,v_1)})$ is finite for some $(u_1,v_1)\in \mathcal{NE}$.
		\item $P^2_M(u,v) = \mu_j(P,M_{(u,v)}) $  for all  $u \ll 0,  v \gg 0$, if $\mu_j(P,M_{(u_1,v_1)})$ is finite for some $(u_1,v_1)\in \mathcal{NW^*}$.
		\item $P^3_M(u,v) = \mu_j(P,M_{(u,v)}) $   for all  $  u,v \ll 0$, if $\mu_j(P,M_{(u_1,v_1)})$ is finite for some $(u_1,v_1)\in \mathcal{S^*W^*}$.
		\item $P^4_M(u,v) = \mu_j(P,M_{(u,v)})$  for all  $  u \gg 0, v \ll 0$, if $\mu_j(P,M_{(u_1,v_1)})$ is finite for some $(u_1,v_1)\in  \mathcal{S^*E}$.
	\end{enumerate}
\end{theorem}
\begin{proof}
From the proof of Theorem \ref{bi-bass}, we have $\mu_j(P,M_{(u,v)})= \alpha_{(u,v)}=\dim_K V_{(u,v)}$, where $V=H_d(Y_1,\ldots, Y_d; L)$ is a bigraded generalized Eulerian holonomic $A_{n,m}(k)$-module. The result now follows from Theorem \ref{bi-poly}.
\end{proof}

\s{\bf Associated primes}.

 \begin{theorem}[With hypotheses as in \ref{sa_application}] \label{bi-ass} Further assume that either $C$ is a regular local ring or a smooth affine algebra over a field $K$ of characteristic zero. Then $\bigcup_{(u,v) \in \Z^2} \Ass_C M_{(u,v)}   $ is a finite set.

Moreover, if we pick $(u,v)$ from one of the blocks $\{\mathcal{NE}, \mathcal{NW^*}, \mathcal{S^*E}, \mathcal{S^*W^*}\}$, then $\Ass_C M_{(u,v)}=\Ass_C M_{(u_0,v_0)}$, where $(u_0,v_0) \in \Z^2$ denotes the corners of the respective blocks.
\end{theorem}

\begin{proof}
The first statement follows using the same line of proof as in \cite[Theorem 12.4 (1)]{TP2}.

\vspace{0.1cm}
Next, let 
\[
\bigcup_{(u,v) \in \Z^2} \Ass_C M_{(u,v)}  = \{ P_1, \ldots, P_l \}.
\]

\begin{minipage}{0.5\textheight}
Let $P = P_i$ for some $i$. Then by Theorem \ref{bi-bass},
\begin{enumerate}[(i)]
	\item $\mu_0(P, M_{(0,0)}) > 0$ if and only if $\mu_0(P, M_{(u,v)}) > 0$ for all $(u,v) \in \mathcal{NE}$. 
	\item $\mu_0(P, M_{(-n,0)}) > 0$ if and only if $\mu_0(P, M_{(u,v)}) > 0$ for all $(u,v) \in \mathcal{NW^*}$. 
	\item $\mu_0(P, M_{(-n,-m)}) > 0$ if and only if $\mu_0(P, M_{(u,v)}) > 0$ for all $(u,v) \in \mathcal{S^*W^*}$. 
	\item $\mu_0(P, M_{(0,-m)}) > 0$ if and only if $\mu_0(P, M_{(u,v)}) > 0$ for all $(u,v) \in  \mathcal{S^*E}$. 
\end{enumerate}  
	\end{minipage}
\begin{minipage}{0.175\textheight}
\begin{center}
	\begin{tikzpicture}[scale=0.15]
	\draw[->] (-8.5,0)--(8.5,0) node[right]{$u$};
	\draw[->] (0,-8.5)--(0,8.5) node[above]{$v$};
	\draw[dashdotted, red] (-2,8.5)--(-2,-8.5);
	\draw[dashdotted, red] (-8.5,-3)--(8.5, -3);
	\draw[fill=cyan,fill opacity=0.35,draw=none] (0,8.5)--(0,0)--(8.5,0)--(8.5,8.5)--(0,8.5);
	\draw[fill=cyan,fill opacity=0.35,draw=none] (-2,8.5)--(-2,0)--(-8.5,0)--(-8.5,8.5)--(-2,8.5);
	\draw[fill=cyan,fill opacity=0.35,draw=none] (-2,-8.5)--(-2,-3)--(-8.5,-3)--(-8.5,-8.5)--(-2,-8.5);
	\draw[fill=cyan,fill opacity=0.35,draw=none] (0,-8.5)--(0,-3)--(8.5,-3)--(8.5,-8.5)--(0,-8.5);
	\node[blue,draw=none] at (3,-5) {\tiny $(0,-m)$};
	\node[red] at (0,-3){\tiny$\bullet$};
	\node[blue,draw=none] at (2.5,1.5) {\tiny $(0,0)$};
	\node[red] at (0,0){\tiny$\bullet$};
	\node[blue,draw=none] at (-5,1.5) {\tiny $(-n,0)$};
	\node[red] at (-2,0){\tiny$\bullet$};
	\node[blue,draw=none] at (-5.5,-5) {\tiny $(-n,-m)$};
	\node[red] at (-2,-3){\tiny$\bullet$};
	\end{tikzpicture}
\end{center}
\end{minipage}

\vspace{0.2cm}
This establishes the result.	
\end{proof}

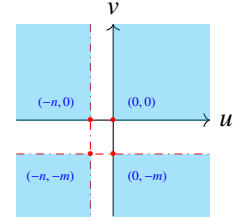
\s{\bf Dimensions of supports and injective dimensions}.

\begin{lemma}[With hypotheses as in \ref{sa_application}]\label{injdim-dim} Then for all $(u,v) \in \Z^2$,
	\[
	\injdim M_{(u,v)} \leq \dim M_{(u,v)}.
	\]
\end{lemma}
\begin{proof}
	Let $P$ be a prime ideal in $C$. Then from Remark \ref{rmk_Bass} we have
	\[
	\mu_j(P, M_{(u,v)}) = \mu_0(P, H^j_P(M_{(u,v)})).
	\]
	By Grothendieck's vanishing theorem, $H^j_P(M_{(u,v)}) = 0$ for all $j > \dim M_{(u,v)}$, see \cite[6.1.2]{BS}.  So $\mu_j(P, M_{(u,v)}) = 0$ for every $j > \dim M_{(u,v)}$. The statement follows.
\end{proof}

\begin{theorem}[With hypotheses as in \ref{sa_application} and notations as in Subsection \ref{notations}]\label{injdim-and-dim-gen} Pick a region $\mathcal{R}$ from the set \[\{\mathcal{NE}, \mathcal{NW^*}, \mathcal{S^*W^*}, \mathcal{S^*E}, \mathrm{Trun}(\mathcal{N}), \mathrm{Trun}(\mathcal{W^*}), \mathrm{Trun}(\mathcal{S^*}), \mathrm{Trun}(\mathcal{E})\}.\] Then $\injdim M_{(u,v)}$ and $\dim M_{(u,v)}$ remain unchanged for all $(u,v)\in \mathcal{R}$. In particular, if $\mathcal{R}$ is one of the blocks $\big\{\mathcal{NE}, \mathcal{NW^*}, \mathcal{S^*W^*}, \mathcal{S^*E}\big\}$, then 
\[\injdim M_{(u,v)}=\injdim M_{(u_0,v_0)} \quad  \mbox{ and } \quad \dim M_{(u,v)}=\dim M_{(u_0,v_0)}\] for all $(u,v)\in \mathcal{R}$, where $(u_0,v_0)$ denotes the corner of the respective block.
\end{theorem}
\begin{proof}
The first assertion follows from Lemma \ref{injdim-dim}.

\vspace{0.1cm}	
	Next, let $P$ be a prime ideal in $C$.  Let $(u,v)\in \mathcal{NE}$.
	
	\vspace{0.1cm}
	Fix $j \geq 0$. From Theorem \ref{bi-bass} we get that $\mu_j(P, M_{(u,v)})> 0$ if and only if $\mu_j(P, M_{(0,0)})> 0$. Thus 
	\[\injdim M_{(u,v)}=\injdim M_{(0,0)} \quad \mbox{ for all } (u,v) \in \mathcal{NE}.\]
	
Further, observe that $\FF(-)_P$ is a bigraded Lyubeznik functor on ${}^*\Mod(C_P[X_1,\ldots, X_n, Y_1, \ldots, Y_m])$. By Theorem \ref{bi-rigid}, $(M_{(0,0)})_P \neq 0$ if and only if $(M_{(u,v)})_P \neq 0$. Hence 
\[\dim M_{(u,v)}=\dim M_{(0,0)} \quad \mbox{ for all } (u,v) \in \mathcal{NE}.\]
	
	\vspace{0.1cm}
	The respective statements for the other regions can be obtained using same arguments as above.
\end{proof}	

\begin{theorem}[With hypotheses as in \ref{sa_application} and notations as in Subsection \ref{notations}]
Let $n,m \geq 2$. Pick $(u,v) \in \Z^2$ such that $-n<u<0$ and $-m< v<0$. Then 
\begin{enumerate}[\rm(i)]
\item $\injdim M_{(u_1,v_1)} \leq \min \{ \injdim M_{(0,0)}, \injdim M_{(-n,0)}, \injdim M_{(-n,-m)}, \injdim M_{(0,-m)} \},$
\item $\dim M_{(u_1,v_1)} \leq \min \{ \dim M_{(0,0)}, \dim M_{(-n,0)}, \dim M_{(-n,-m)}, \dim M_{(0,-m)} \}$.
\end{enumerate}

Furthermore, suppose that $\mathcal{R}$ is one of the truncated regions $\{\mathrm{Trun}(\mathcal{N}), \mathrm{Trun}(\mathcal{W^*}), \mathrm{Trun}(\mathcal{S^*}), \mathrm{Trun}(\mathcal{E})\}$. Then 
\[\injdim M_{(u,v)} \leq \min \{ \injdim M_{(u_1,v_1)}, \injdim M_{(u_2,v_2)}\} \quad \mbox{ and } \quad \dim M_{(u,v)} \leq \min \{ \dim M_{(u_1,v_1)}, \dim M_{(u_2,v_2)}\}\]	
for all $(u,v)\in \mathcal{R}$, where $\{(u_1,v_1), (u_2,v_2)\}$ denotes the pair of corners of the respective regions.
\end{theorem}

\begin{proof}
From Theorem \ref{bi-bass} we get that if $\mu_j(P, M_{(u_1, v_1)})> 0$ then all the Bass numbers $\mu_j(P, M_{(0,0)})$, $\mu_j(P, M_{(-n,0)})$, $\mu_j(P, M_{(-n,-m)}), \mu_j(P, M_{(0,-m)})> 0$. Thus (i) follows.
	
	\vspace{0.1cm}
As $\FF(-)_P$ is a bigraded Lyubeznik functor on ${}^*\Mod(C_P[X_1,\ldots, X_n, Y_1, \ldots, Y_m])$ so by Theorem \ref{bi-rigid2}, $(M_{(u_1,v_1)})_P \neq 0$  implies $(M_{(0,0)})_P$, $(M_{(-n,0)})_P,$  $(M_{(-n,-m)})_P$ and $(M_{(0,-m)})_P$ all are nonzero. Thus (ii) follows.
	
	\vspace{0.1cm}
One can obtain the assertions for the truncated regions using similar arguments as above.
\end{proof}

\s {\bf Infinite generation}.\\
In this subsection, we give a sufficient condition for infinite generation of components of $H^i_I(R)$. 
\begin{theorem}[With hypothesis as in \ref{sa_application}]
Further assume that $C$ is a domain. Let $I$ be a bihomogeneous ideal in $R$ such that $I \cap C \neq \phi$. Set $M:=H^i_I(R)=\bigoplus_{(u,v)\in \Z^2}M_{(u,v)}$. If $M_{(u,v)} \neq 0$ for some $(u,v)\in \Z^2$, then $M_{(u,v)}$ is not finitely generated as a $C$-module.
\end{theorem}
\begin{proof}
Let $P$ be a prime ideal in $C$. Notice $R_P=C_P[X_1, \ldots, X_n, Y_1, \ldots, Y_m]$. Moreover, 
\[(M_{(u,v)})_P= H_I^i(R)_{(u,v)} \otimes _C C_P= H_I^i(R \otimes_C C_P)_{(u,v)}= H_I^i(R_P)_{(u,v)}.\] 
for every $(u,v)\in \Z^2$. Clearly, if $H_I^i(R_P)_{(u,v)}$ is not a finitely generated $C_P$-module, then $M_{(u,v)}$ is also not a finitely generated $C$-module. Therefore, it is enough to prove the result under the assumption that $C$ is a local ring.
	
If possible, suppose that $M_{(u,v)}$ is finitely generated as a $C$-module. Then $\depth M_{(u,v)} \leq \dim M_{(u,v)}$. From Theorem \ref{injdim-dim} we have $\injdim M_{(u,v)} \leq \dim M_{(u,v)}$. So by \cite[Theorem 3.1.17]{BH}, $\dim M_{(u,v)} \leq \injdim M_{(u,v)}= \depth C=\dim C$. Together we get that $\dim M_{(u,v)} \leq \dim C= \injdim M_{(u,v)} \leq \dim M_{(u,v)}$, that is, $\dim M_{(u,v)}= \dim C$. Thus $M_{(u,v)}$ is a maximal Cohen Macaulay $C$-module. Let $0 \neq c \in I \cap C$. Since $C$ is a domain, $c^t$ is $B$-regular and hence $M_{(u,v)}$-regular for all $t \geq 1$. But as $M_{(u,v)}$ is $I$-torsion so $(0:_{M_{(u,v)}} c^t) \neq 0$ for some $t$, leading to a contradiction.
\end{proof}

\s {\bf Hilbert series}.\\
We now present a Terai-type formula (see \cite{NT}, \cite[p. 19]{JAM_LC}) for the bigraded Hilbert series of $H^i_I(R)$.

\begin{theorem}[With hypothesis as in \ref{sa_application}]\label{bi_Hilbert} Further assume that $C=K$ is a field. Let $M_{(u,v)}=0$ for all $(u,v)\in \Z^2$ such that either $-n<u<0$ or $-m<v<0$, but $M \neq 0$. Then the bigraded Hilbert series of $M$ is given by
\begin{align*}
H_M(t_1,t_2)=\frac{d_1}{(1-t_1)^n(1-t_2)^m}+\frac{d_2t_1^{-n}}{(1-t_1^{-1})^n(1-t_2)^m}+\frac{d_3t_1^{-n}t_2^{-m}}{(1-t_1^{-1})^n(1-t_2^{-1})^m}+\frac{d_4t_2^{-m}}{(1-t_1)^n(1-t_2)^m},
\end{align*}
where $d_1=\dim_K M_{(0,0)}$, $d_2=\dim_K M_{(-n,0)}$, $d_3=\dim_K M_{(-n,-m)}$, and $d_4=\dim_K M_{(0,-m)}$.
\end{theorem}

\begin{note*}
	
\noindent	
\begin{minipage}[t]{0.425\textheight}
Clearly, $M_{(u,v)}=0$ for all $(u,v)$ lies outside the shaded part in Figure $(ix)$. Note that we are considering all those modules $M$ whose nonzero components are represented by combination of shaded parts in Figures (I), (II), (III) and (IV) in Theorem \ref{bi-rigid}.
\end{minipage}
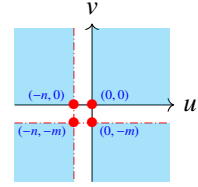
\begin{minipage}[c][7.75ex][t]{0.25\textheight} 
\begin{center}
	\begin{tikzpicture}[scale=0.12]
	\draw[->] (-8.5,0)--(9,0) node[right]{$u$};
	\draw[->] (0,-8.5)--(0,9) node[above]{$v$};
	\draw[densely dashdotted, red] (-2,8.5)--(-2,-8.5);
	\draw[densely dashdotted, red] (-8.5,-2)--(8.5, -2);
	\draw[fill=cyan,fill opacity=0.35,draw=none] (0,8.5)--(0,0)--(8.5,0)--(8.5,8.5)--(0,8.5);
	\draw[fill=cyan,fill opacity=0.35,draw=none] (-2,8.5)--(-2,0)--(-8.5,0)--(-8.5,8.5)--(-2,8.5);
	\draw[fill=cyan,fill opacity=0.35,draw=none] (-2,-8.5)--(-2,-2)--(-8.5,-2)--(-8.5,-8.5)--(-2,-8.5);
	\draw[fill=cyan,fill opacity=0.35,draw=none] (0,-8.5)--(0,-2)--(8.5,-2)--(8.5,-8.5)--(0,-8.5);
	\node at (0,-2){\textcolor{red}{$\bullet$}};
	\node at (0,0){\textcolor{red}{$\bullet$}};
	\node at (-2,0){\textcolor{red}{$\bullet$}};
	\node at (-2,-2){\textcolor{red}{$\bullet$}};
	\node[blue,draw=none] at (3,-3) {\tiny $(0,-m)$};
	\node[blue,draw=none] at (2.5,1) {\tiny $(0,0)$};
	\node[blue,draw=none] at (-5,1) {\tiny $(-n,0)$};
	\node[blue,draw=none] at (-5.5,-3) {\tiny $(-n,-m)$};
	\end{tikzpicture}	
\end{center}
\end{minipage}
\end{note*}

\begin{proof}
The result follows from Proposition \eqref{msudan}.
\end{proof}

We now discuss some examples.

\begin{example}
Let $R=K[X_1,X_2,Y_1,Y_2]$ be a bigraded polynomial ring over a field $K$ of characteristic zero with $\bideg X_i=(1,0)$ and $\bideg Y_j=(0,1)$ for $1 \leq i,j \leq 2$. Take the ideal $I=(X_1Y_1,X_1Y_2,X_2Y_2,X_2Y_2)$ in $R$. As 
\begin{align*}
&(X_2Y_1)^2=(X_1Y_2+X_2Y_1)X_2Y_1-X_1Y_1X_2Y_2,\\
 \mbox{ and } \quad &(X_1Y_2)^2=(X_1Y_2+X_2Y_1)X_1Y_2-X_1Y_1X_2Y_2,
\end{align*}
so we have $\sqrt{(X_1Y_1,X_1Y_2+X_2Y_1,X_2Y_2)}=\sqrt{I}$.
Notice $\sqrt{I}=\sqrt{(X_1,X_2) \cap (Y_1,Y_2)}=(X_1,X_2) \cap (Y_1,Y_2)$. Set $J=(X_1Y_1,X_1Y_2+X_2Y_1,X_2Y_2)$, $\mathfrak{a}=(X_1,X_2)$, and $\mathfrak{b}=(Y_1,Y_2)$. By Example \ref{eg_rigidity},

\noindent
\begin{minipage}{0.525\textheight}
\begin{align*}
H^2_J(R) = H^2_{I}(R)\cong & H^2_{\mathfrak{a}}(R)\oplus H^2_{\mathfrak{b}}(R)\\
= & X_1^{-1}X_2^{-1}K[X_1^{-1},X_2^{-1},Y_1,Y_2]\oplus Y_1^{-1}Y_2^{-1}K[X_1,X_2,Y_1^{-1},Y_2^{-1}],
\end{align*}
which nonzero components can be represented by the shaded regions of the figure. Therefore, from Theorem \ref{bi_Hilbert} we get that
\end{minipage}
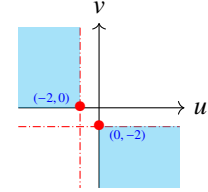
\begin{minipage}{0.2\textheight}
\begin{center}
	\begin{tikzpicture}[scale=0.125]
	\draw[->] (-8.5,0)--(9,0) node[right]{$u$};
	\draw[->] (0,-8.5)--(0,9) node[above]{$v$};
	\draw[densely dashdotted, red] (-2,8.5)--(-2,-8.5);
	\draw[densely dashdotted, red] (-8.5,-2)--(8.5, -2);
	\draw[fill=cyan,fill opacity=0.35,draw=none] (-2,8.5)--(-2,0)--(-8.5,0)--(-8.5,8.5)--(-2,8.5);
	\draw[fill=cyan,fill opacity=0.35,draw=none] (0,-8.5)--(0,-2)--(8.5,-2)--(8.5,-8.5)--(0,-8.5);
	\node at (0,-2){\textcolor{red}{\textbullet}};
	\node at (-2,0){\textcolor{red}{\textbullet}};
	\node[blue,draw=none] at (3,-3) {\tiny $(0,-2)$};
	\node[blue,draw=none] at (-5,1) {\tiny $(-2,0)$};
	\end{tikzpicture}	
\end{center}
\end{minipage}

\begin{align*}
H_{H^2_J(R)}(t_1, t_2)
=&\dim_K M_{(-2,0)} \cdot\frac{t_1^{-2}}{(1-t_1^{-1})^2}\cdot \frac{1}{(1-t_2)^2}+\dim_K M_{(0,-2)} \cdot\frac{1}{(1-t_1)^2}\cdot \frac{t_2^{-2}}{(1-t_2^{-1})^2}\\
=& \frac{t_1^{-2}}{(1-t_1^{-1})^2}\cdot \frac{1}{(1-t_2)^2}+\frac{1}{(1-t_1)^2}\cdot \frac{t_2^{-2}}{(1-t_2^{-1})^2}\\
=&\frac{2}{(1-t_1)^2(1-t_2)^2},
\end{align*}
since $\dim_K M_{(-2,0)}=1=\dim_K M_{(0,-2)}$.
\end{example}

We now apply our results on a binomial edge ideal.
\begin{example}
	Let $R=K[W_1, W_2, W_3, W_4, W_5, W_6]$, and $I=I_2(A)$ be an ideal in $R$ generated by $2 \times 2$-minor of the matrix 
	$A=
	\left( {\begin{array}{ccc}
		W_1 & W_2 & W_3 \\
		W_4 & W_5 & W_6 \\
		\end{array} } \right)$.
	If we rename the variables by $W_i=X_i$ for all $1 \leq i \leq 3$, and $W_j=Y_j$ for all $4 \leq j \leq 6$. Clearly then $I$ is the binomial edge ideal associated to the $3$-cycle complete graph $K_3$. Set $\bideg X_i=(1,0)$ and $\bideg Y_j=(0,1)$ for $1\leq i,j \leq 3$. Then $I$ is a bihomogeneous ideal. In \cite[Example 6.1]{UW}, Walther showed that $H^3_I(R) \cong E_R(K)$. Hence from Theorem \ref{bi_Hilbert} and Example \ref{eg_rigidity} we get
	\begin{equation*}
	\begin{split}
	H(H^3_I(R);(t_1, t_2, t_3))=\dim_K H^3_I(R)_{(-3,-3)} \cdot \frac{t_1^{-3}}{(1-t_1^{-1})^{3}}\cdot  \frac{t_2^{-3}}{(1-t_2^{-1})^{3}}=\frac{t_1^{-3}t_2^{-3}}{(1-t_1^{-1})^{3}(1-t_2^{-1})^{3}}.
	\end{split}
	\end{equation*}
\end{example}


\begin{thebibliography}{15}
	
	\bibitem{JAM_LC}
	J. \`Alvarez. Montaner,
	\emph{Three lectures on local cohomology modules supported on monomial ideals},
	\href{https://upcommons.upc.edu/handle/2117/16881}{2012}.

	
	\bibitem{Mon20}
	J. \`Alvarez Montaner, 
	\emph{Local cohomology of binomial edge ideals and their generic initial ideals}, Collect. Math., Vol. 71 (2020), No. 2, 331--348.
	
	
	
	\bibitem{BJ}
	J.-E. Bj\"{o}rk, \emph{Rings of differential operators}, North-Holland Math. Library \textbf{21}, North Holland, Amsterdam, 1979. 
		

	\bibitem{BS}
	M. P. Brodmann and R. Y. Sharp, 
	\emph{Local cohomology: an algebraic introduction with geometric applications}, Cambridge Studies in Advanced Mathematics, Vol. 60, Cambridge University Press, Cambridge, 1998.  
	

	
	\bibitem{BS2}
	M. P. Brodmann, R. Y. Sharp,
	\emph{Supporting degrees of multi-graded local cohomology modules},
	J. Algebra, Vol. 321 (2009), No. 2, 450--482.
	
	\bibitem{BH}
	W.~Bruns and J.~Herzog, \emph{{Cohen-Macaulay rings}}, Cambridge Studies in Advanced Mathematics, Vol. 39, Cambridge University Press,~Cambridge, 1993.

	\bibitem{GCN}
G. Colom{\'e} i Nin, 
\emph{Multigraded structure and the depth of Blow-up algebras}, Universitat de Barcelona, 2008, \href{https://www.tdx.cat/bitstream/handle/10803/663/GCN_THESIS.pdf?sequence=9.xml}{Thesis}.

\bibitem{Cox}
D. Cox, 
\emph{The homogeneous coordinate ring of a toric variety}, J. Algebraic Geom., Vol 4 (1995), 17--50.


\bibitem{HHHKR}
J. Herzog, T. Hibi, F. Hreinsd\'ottir, T.Kahle, and J. Rauh, \emph{Binomial edge ideals and conditional independence statements}, Adv. Appl. Math., Vol. 45 (2010), No. 3, 317--333.
	
	
	\bibitem{24Hrs}
	S. B. Iyengar, G. J. Leuschke, A. Leykin, C. Miller, E. Miller, A. K. Singh and U. 	Walther, \emph{Twenty-four hours of local cohomology}, Graduate Studies in Mathematics, American Mathematical Society, Vol. 87, 2007.
	
	\bibitem{L23}
	L. Li,
	\emph{Nakajima's quiver varieties and triangular bases of rank-2 cluster algebras}, J. Algebra, Vol. 634 (2023), 97--164.
	
	\bibitem{Lyu1}
	G.~Lyubeznik, 
	\emph{Finiteness Properties of Local Cohomology Modules {\rm(}an Application of $D$-modules to Commutative Algebra{\rm)}},
	Inv. Math., Vol. 113 (1993), 41--55. 
	
	
	\bibitem{MZ}
	L. Ma and W. Zhang, 
	\emph{Eulerian graded D-modules}, Math. Res. Lett. Vol. 21 (2014), No. 1, 149--167.
	
	
	
	\bibitem{Oht11}
	M. Ohtani,
	\emph{Graphs and Ideals Generated by Some $2$-Minors}. Comm. Algebra, Vol. 39 (2011), No. 3, 905--917. 

		
	\bibitem{TP1}
	T. J. Puthenpurakal,
	\emph{de Rahm cohomology of local cohomology modules: The graded case}, Nagoya Math. J., Vol. 217 (2015), 1--21.

	
	\bibitem{TP2}
	T. J. Puthenpurakal,
	\emph{Graded Components of Local Cohomology Modules}, 
	Collect. Math. (2022), 1--37.
	

	
\bibitem{TPSR19}
T. J. Puthenpurakal and S. Roy, 
\emph{Graded components of local cohomology modules of invariant rings}, Comm. Algebra, Vol. 48 (2019), No. 2, 803--814.
		
	\bibitem{TPSR}
	T. J. Puthenpurakal and S. Roy, 
	\emph{Graded components of local cohomology modules II},
Vietnam J. Math., Vol. 52 (2024), No. 1, 1--24.

\bibitem{TPSR24}
T. J. Puthenpurakal and S. Roy, 
Graded components of local cohomology modules of $\mathfrak{C}$-monomial ideals in characteristic zero, to appear in Comm. Algebra, arXiv:2307.03574.
	
	
	\bibitem{TPJS}
	T. J. Puthenpurakal and J. Singh,
	\emph{On derived functors of graded local cohomology modules}, Math. Proc. Camb. Philos. Soc., Vol. 167 (2019), No. 3, 549--565.

	



\bibitem{PS}
P. Schenzel, 
\emph{On the use of local cohomology in algebra and geometry}, in Six lectures on commutative algebra (Bellaterra, 1996), Progr. Math., Vol. 166, 241--292, Birkh¨auser, Basel, 1998.
	
\bibitem{NT}
 N. Terai, 
\emph{Local cohomology modules with respect to monomial ideals}, 
Proceedings of the 20th Symposium on Commutative Ring Theory (1998), 181 - 189.


\bibitem{UW}
U. Walther,
\emph{Algorithmic computation of local cohomology modules and the local cohomological dimension of algebraic varieties},
J. Pure Appl. Algebra, Vol. 139 (1999), Issues 1-3, 303 - 321.


\bibitem{Wil23_1}
D. Williams,
\emph{LF-Covers and Binomial Edge Ideals of K\"onig Type},
preprint: arXiv:2310.14410v4.

\bibitem{Wil23_2}
D. Williams,
\emph{The Local Cohomology Modules of the Binomial Edge Ideals of the Complements of Connected Graphs of Girth at Least 5},
preprint: arXiv:2312.01613v3.	
\end{thebibliography}
\end{document}